\documentclass[a4paper,11pt,reqno]{amsart}

\usepackage[utf8]{inputenc}
\usepackage[T1]{fontenc}
\usepackage{hyperref}
\usepackage{xcolor}
\hypersetup{
    colorlinks,
    linkcolor={red!50!black},
    citecolor={blue!50!black},
    urlcolor={blue!80!black}
}

\usepackage{microtype}
\usepackage{MnSymbol}
\usepackage{textcomp}

\usepackage[mathcal]{euscript}
\let\mathcaltmp\mathcal
\usepackage{mathrsfs}
\let\mathcal\mathscr
\let\mathscr\mathcaltmp

\makeatletter
\def\thm@space@setup{\thm@preskip=7pt
\thm@postskip=7pt}
\makeatother
\newtheoremstyle{plain}
  {}
  {}
  {\slshape}
  {}
  {\bfseries}
  {.}
  { }
  {}

\newtheoremstyle{definition}
  {}
  {}
  {}
  {}
  {\bfseries}
  {.}
  { }
  {}

\makeatletter
\renewenvironment{proof}[1][\proofname]{\par
  \pushQED{\qed}%
  \normalfont \topsep0\p@\relax
  \trivlist
  \item[\hskip\labelsep\itshape
  #1\@addpunct{.}]\ignorespaces
}{%
  \popQED\endtrivlist\@endpefalse
}
\makeatother

\makeatletter
\newcommand{\eqnum}{\refstepcounter{equation}\textup{\tagform@{\theequation}}}
\makeatother

\makeatletter \@addtoreset{equation}{section} \makeatother
\renewcommand{\theequation}{\thesection.\arabic{equation}}

\newtheorem{thm}[equation]{Theorem}
\newtheorem{thmX}{Theorem}

\newtheorem{lem}[equation]{Lemma}
\newtheorem{cor}[equation]{Corollary}
\newtheorem{prop}[equation]{Proposition}
\newtheorem*{defthm*}{Definition/Theorem}

\theoremstyle{definition}
\newtheorem{defn}[equation]{Definition}
\newtheorem{rem}[equation]{Remark}
\newtheorem{exam}[equation]{Example}

\newtheorem{notat}[equation]{Notation}

\newtheorem*{exam*}{Example}

\newcommand\arXiv[1]{\href{http://arxiv.org/abs/#1}{arXiv:#1}}

\usepackage{xparse}

\usepackage{xspace}

\newcommand{\changelocaltocdepth}[1]{%
  \addtocontents{toc}{\protect\setcounter{tocdepth}{#1}}%
  \setcounter{tocdepth}{#1}}

\newcommand{\nc}{\newcommand}
\nc{\renc}{\renewcommand}
\nc{\ssec}{\subsection}
\nc{\sssec}{\subsubsection}
\nc{\on}{\operatorname}
\nc{\term}[1]{#1\xspace}

\usepackage{tikz}
\usetikzlibrary{matrix}
\usepackage{tikz-cd}

\tikzset{
  commutative diagrams/.cd,
  arrow style=tikz,
  diagrams={>=latex}}
\makeatletter
\tikzset{
  column sep/.code=\def\pgfmatrixcolumnsep{\pgf@matrix@xscale*(#1)},
  row sep/.code   =\def\pgfmatrixrowsep{\pgf@matrix@yscale*(#1)},
  matrix xscale/.code=%
    \pgfmathsetmacro\pgf@matrix@xscale{\pgf@matrix@xscale*(#1)},
  matrix yscale/.code=%
    \pgfmathsetmacro\pgf@matrix@yscale{\pgf@matrix@yscale*(#1)},
  matrix scale/.style={/tikz/matrix xscale={#1},/tikz/matrix yscale={#1}}}
\def\pgf@matrix@xscale{1}
\def\pgf@matrix@yscale{1}
\makeatother

\usepackage[inline]{enumitem}
\setlist[enumerate,1]{label={(\alph*)},itemsep=\parskip}

\newlist{thmlist}{enumerate}{1}
\setlist[thmlist,1]{
  label={\em(\roman*)}, ref={(\roman*)},
  itemsep=0.5em,
  topsep=0em,
  leftmargin=*,
  align=left,widest=vi)}

\newlist{thmlistbis}{enumerate}{1}
\setlist[thmlistbis,1]{
  label={\em(\roman*~\textit{bis})},
  ref={(\roman*}~\textit{bis}\upshape{)},
  itemsep=0.5em,
  topsep=-0.7em,
  leftmargin=0pt, align=right, widest=vi)}

\newlist{defnlist}{enumerate}{2}
\setlist[defnlist,1]{
  label={(\roman*)}, ref={(\roman*)},
  itemsep=0.5em,
  topsep=0em,
  leftmargin=*,
  align=left, widest=vi)}

\setlist[defnlist,2]{
  label={(\alph*)}, ref={(\alph*)},
  itemsep=0.75em,
  labelsep=0em,labelindent=0em,leftmargin=*,align=left,widest=vi),
  topsep=0.75em}

\newlist{defnlistbis}{enumerate}{1}
\setlist[defnlistbis,1]{
  label={(\roman*~\textit{bis})},
  ref={(\roman*}~\textit{bis}\upshape{)},
  itemsep=0.5em,
  topsep=0em,
  leftmargin=*,
  align=left, widest=vi)}

\newlist{inlinelist}{enumerate*}{1}
\setlist[inlinelist,1]{label={(\alph*)}}

\newlist{inlinedefnlist}{enumerate*}{1}
\definecolor{green}{HTML}{38550C}
\setlist[inlinedefnlist,1]{label={\color{green}(\roman*)}}

\newlist{inlinethmlist}{enumerate*}{1}
\definecolor{green}{HTML}{38550C}
\setlist[inlinethmlist,1]{label={\color{green}(\roman*)}}

\nc{\cA}{\ensuremath{\mathcal{A}}\xspace}
\nc{\cB}{\ensuremath{\mathcal{B}}\xspace}
\nc{\cC}{\ensuremath{\mathcal{C}}\xspace}
\nc{\cD}{\ensuremath{\mathcal{D}}\xspace}
\nc{\cE}{\ensuremath{\mathcal{E}}\xspace}
\nc{\cF}{\ensuremath{\mathcal{F}}\xspace}
\nc{\cG}{\ensuremath{\mathcal{G}}\xspace}
\nc{\cH}{\ensuremath{\mathcal{H}}\xspace}
\nc{\cI}{\ensuremath{\mathcal{I}}\xspace}
\nc{\cJ}{\ensuremath{\mathcal{J}}\xspace}
\nc{\cK}{\ensuremath{\mathcal{K}}\xspace}
\nc{\cL}{\ensuremath{\mathcal{L}}\xspace}
\nc{\cM}{\ensuremath{\mathcal{M}}\xspace}
\nc{\cN}{\ensuremath{\mathcal{N}}\xspace}
\nc{\cO}{\ensuremath{\mathcal{O}}\xspace}
\nc{\cP}{\ensuremath{\mathcal{P}}\xspace}
\nc{\cQ}{\ensuremath{\mathcal{Q}}\xspace}
\nc{\cR}{\ensuremath{\mathcal{R}}\xspace}
\nc{\cS}{\ensuremath{\mathcal{S}}\xspace}
\nc{\cT}{\ensuremath{\mathcal{T}}\xspace}
\nc{\cU}{\ensuremath{\mathcal{U}}\xspace}
\nc{\cV}{\ensuremath{\mathcal{V}}\xspace}
\nc{\cW}{\ensuremath{\mathcal{W}}\xspace}
\nc{\cX}{\ensuremath{\mathcal{X}}\xspace}
\nc{\cY}{\ensuremath{\mathcal{Y}}\xspace}
\nc{\cZ}{\ensuremath{\mathcal{Z}}\xspace}

\nc{\sA}{\ensuremath{\mathscr{A}}\xspace}
\nc{\sB}{\ensuremath{\mathscr{B}}\xspace}
\nc{\sC}{\ensuremath{\mathscr{C}}\xspace}
\nc{\sD}{\ensuremath{\mathscr{D}}\xspace}
\nc{\sE}{\ensuremath{\mathscr{E}}\xspace}
\nc{\sF}{\ensuremath{\mathscr{F}}\xspace}
\nc{\sG}{\ensuremath{\mathscr{G}}\xspace}
\nc{\sH}{\ensuremath{\mathscr{H}}\xspace}
\nc{\sI}{\ensuremath{\mathscr{I}}\xspace}
\nc{\sJ}{\ensuremath{\mathscr{J}}\xspace}
\nc{\sK}{\ensuremath{\mathscr{K}}\xspace}
\nc{\sL}{\ensuremath{\mathscr{L}}\xspace}
\nc{\sM}{\ensuremath{\mathscr{M}}\xspace}
\nc{\sN}{\ensuremath{\mathscr{N}}\xspace}
\nc{\sO}{\ensuremath{\mathscr{O}}\xspace}
\nc{\sP}{\ensuremath{\mathscr{P}}\xspace}
\nc{\sQ}{\ensuremath{\mathscr{Q}}\xspace}
\nc{\sR}{\ensuremath{\mathscr{R}}\xspace}
\nc{\sS}{\ensuremath{\mathscr{S}}\xspace}
\nc{\sT}{\ensuremath{\mathscr{T}}\xspace}
\nc{\sU}{\ensuremath{\mathscr{U}}\xspace}
\nc{\sV}{\ensuremath{\mathscr{V}}\xspace}
\nc{\sW}{\ensuremath{\mathscr{W}}\xspace}
\nc{\sX}{\ensuremath{\mathscr{X}}\xspace}
\nc{\sY}{\ensuremath{\mathscr{Y}}\xspace}
\nc{\sZ}{\ensuremath{\mathscr{Z}}\xspace}

\nc{\bA}{\ensuremath{\mathbf{A}}\xspace}
\nc{\bB}{\ensuremath{\mathbf{B}}\xspace}
\nc{\bC}{\ensuremath{\mathbf{C}}\xspace}
\nc{\bD}{\ensuremath{\mathbf{D}}\xspace}
\nc{\bE}{\ensuremath{\mathbf{E}}\xspace}
\nc{\bF}{\ensuremath{\mathbf{F}}\xspace}
\nc{\bG}{\ensuremath{\mathbf{G}}\xspace}
\nc{\bH}{\ensuremath{\mathbf{H}}\xspace}
\nc{\bI}{\ensuremath{\mathbf{I}}\xspace}
\nc{\bJ}{\ensuremath{\mathbf{J}}\xspace}
\nc{\bK}{\ensuremath{\mathbf{K}}\xspace}
\nc{\bL}{\ensuremath{\mathbf{L}}\xspace}
\nc{\bM}{\ensuremath{\mathbf{M}}\xspace}
\nc{\bN}{\ensuremath{\mathbf{N}}\xspace}
\nc{\bO}{\ensuremath{\mathbf{O}}\xspace}
\nc{\bP}{\ensuremath{\mathbf{P}}\xspace}
\nc{\bQ}{\ensuremath{\mathbf{Q}}\xspace}
\nc{\bR}{\ensuremath{\mathbf{R}}\xspace}
\nc{\bS}{\ensuremath{\mathbf{S}}\xspace}
\nc{\bT}{\ensuremath{\mathbf{T}}\xspace}
\nc{\bU}{\ensuremath{\mathbf{U}}\xspace}
\nc{\bV}{\ensuremath{\mathbf{V}}\xspace}
\nc{\bW}{\ensuremath{\mathbf{W}}\xspace}
\nc{\bX}{\ensuremath{\mathbf{X}}\xspace}
\nc{\bY}{\ensuremath{\mathbf{Y}}\xspace}
\nc{\bZ}{\ensuremath{\mathbf{Z}}\xspace}

\nc{\bbA}{\ensuremath{\mathbb{A}}\xspace}
\nc{\bbB}{\ensuremath{\mathbb{B}}\xspace}
\nc{\bbC}{\ensuremath{\mathbb{C}}\xspace}
\nc{\bbD}{\ensuremath{\mathbb{D}}\xspace}
\nc{\bbE}{\ensuremath{\mathbb{E}}\xspace}
\nc{\bbF}{\ensuremath{\mathbb{F}}\xspace}
\nc{\bbG}{\ensuremath{\mathbb{G}}\xspace}
\nc{\bbH}{\ensuremath{\mathbb{H}}\xspace}
\nc{\bbI}{\ensuremath{\mathbb{I}}\xspace}
\nc{\bbJ}{\ensuremath{\mathbb{J}}\xspace}
\nc{\bbK}{\ensuremath{\mathbb{K}}\xspace}
\nc{\bbL}{\ensuremath{\mathbb{L}}\xspace}
\nc{\bbM}{\ensuremath{\mathbb{M}}\xspace}
\nc{\bbN}{\ensuremath{\mathbb{N}}\xspace}
\nc{\bbO}{\ensuremath{\mathbb{O}}\xspace}
\nc{\bbP}{\ensuremath{\mathbb{P}}\xspace}
\nc{\bbQ}{\ensuremath{\mathbb{Q}}\xspace}
\nc{\bbR}{\ensuremath{\mathbb{R}}\xspace}
\nc{\bbS}{\ensuremath{\mathbb{S}}\xspace}
\nc{\bbT}{\ensuremath{\mathbb{T}}\xspace}
\nc{\bbU}{\ensuremath{\mathbb{U}}\xspace}
\nc{\bbV}{\ensuremath{\mathbb{V}}\xspace}
\nc{\bbW}{\ensuremath{\mathbb{W}}\xspace}
\nc{\bbX}{\ensuremath{\mathbb{X}}\xspace}
\nc{\bbY}{\ensuremath{\mathbb{Y}}\xspace}
\nc{\bbZ}{\ensuremath{\mathbb{Z}}\xspace}

\nc{\mrm}[1]{\ensuremath{\mathrm{#1}}\xspace}
\nc{\mit}[1]{\ensuremath{\mathit{#1}}\xspace}
\nc{\mbf}[1]{\ensuremath{\mathbf{#1}}\xspace}
\nc{\mcal}[1]{\ensuremath{\mathcal{#1}}\xspace}
\nc{\msc}[1]{\ensuremath{\mathscr{#1}}\xspace}
\nc{\mfr}[1]{\ensuremath{\mathfrak{#1}}\xspace}

\nc{\sub}{\subseteq}
\nc{\too}{\longrightarrow}
\nc{\hook}{\hookrightarrow}
\nc{\hooklongrightarrow}{\lhook\joinrel\longrightarrow}
\nc{\hooklong}{\hooklongrightarrow}
\nc{\hooklongleftarrow}{\longleftarrow\joinrel\rhook}
\nc{\twoheadlongrightarrow}{\relbar\joinrel\twoheadrightarrow}
\nc{\longrightleftarrows}{\ \raisebox{0.3ex}{\(\mathrel{\substack{\xrightarrow{\rule{1em}{0em}} \\[-1ex] \xleftarrow{\rule{1em}{0em}}}}\)}\ }

\renc{\ge}{\geqslant}
\renc{\le}{\leqslant}

\nc{\id}{\mathrm{id}}

\DeclareMathOperator{\Hom}{\on{Hom}}
\nc{\uHom}{\underline{\smash{\Hom}}}
\DeclareMathOperator{\Maps}{\on{Maps}}

\DeclareMathOperator{\End}{\on{End}}

\nc{\uEnd}{\underline{\smash{\End}}}
\DeclareMathOperator{\codim}{\on{codim}}
\nc{\colim}{\varinjlim}
\renc{\lim}{\varprojlim}
\nc{\Cofib}{\on{Cofib}}
\nc{\Fib}{\on{Fib}}
\nc{\initial}{\varnothing}
\nc{\op}{\mathrm{op}}

\DeclareMathOperator*{\fibprod}{\times}

\renc{\setminus}{\smallsetminus}

\usepackage{mathtools}
\DeclarePairedDelimiter\abs{\lvert}{\rvert}%

\newcommand{\thmref}[1]{Theorem~\ref{#1}}

\newcommand{\secref}[1]{Sect.~\ref{#1}}
\newcommand{\ssecref}[1]{Subsect. ~\ref{#1}}
\newcommand{\sssecref}[1]{(\ref{#1})}
\newcommand{\lemref}[1]{Lemma~\ref{#1}}
\newcommand{\propref}[1]{Proposition~\ref{#1}}
\newcommand{\corref}[1]{Corollary~\ref{#1}}

\newcommand{\remref}[1]{Remark~\ref{#1}}
\newcommand{\defnref}[1]{Definition~\ref{#1}}

\renewcommand{\eqref}[1]{(\ref{#1})}

\newcommand{\notatref}[1]{Notation~\ref{#1}}

\newcommand{\itemref}[1]{\ref{#1}}

\nc{\A}{\bA}
\renc{\P}{\bP}
\nc{\V}{\bV}
\nc{\Spec}{\on{Spec}}
\nc{\D}{\on{\mbf{D}}}
\nc{\rD}{\on{\mrm{D}}}
\nc{\Dqc}{\on{\mbf{D}}_{\mrm{qc}}}
\nc{\bDelta}{\mathbf{\Delta}}
\nc{\Cech}{\textnormal{\v{C}}}
\nc{\Dperf}{\on{\mbf{D}}_{\mrm{perf}}}
\nc{\Coh}{\on{Coh}}
\nc{\Qcoh}{\on{Qcoh}}
\nc{\Dcoh}{\on{\mbf{D}}_{\mrm{coh}}}
\nc{\cl}{{\mrm{cl}}}
\nc{\Bl}{\on{Bl}}
\nc{\vir}{\mrm{vir}}
\nc{\Zar}{\mrm{Zar}}
\nc{\et}{\mrm{\acute{e}t}}
\nc{\Nis}{\mrm{Nis}}
\renc{\H}{\on{H}}
\nc{\BM}{\mrm{BM}}
\nc{\Z}{\bZ}
\nc{\Q}{\bQ}
\nc{\K}{{\on{K}}}
\nc{\G}{{\on{G}}}
\nc{\KH}{\mrm{KH}}
\nc{\KGL}{\mrm{KGL}}
\nc{\rightrightrightarrows}
  {\mathrel{\substack{\textstyle\rightarrow\\[-0.6ex]
   \textstyle\rightarrow \\[-0.6ex]
   \textstyle\rightarrow}}}
\nc{\rightrightrightrightarrows}
  {\mathrel{\substack{\textstyle\rightarrow\\[-0.6ex]
   \textstyle\rightarrow \\[-0.6ex]
   \textstyle\rightarrow \\[-0.6ex]
   \textstyle\rightarrow}}}
\nc{\Einfty}{{\sE_\infty}}
\renc{\sp}{\mrm{sp}}
\nc{\Td}{\on{Td}}
\nc{\ch}{\on{ch}}
\nc{\RGamma}{R\Gamma}
\nc{\red}{\mrm{red}}
\nc{\der}{{\mrm{der}}}
\nc{\Mod}{{\mrm{Mod}}}
\nc{\Gr}{{\on{Gr}}}
\nc{\Ind}{\on{Ind}}
\nc{\Pro}{\on{Pro}}
\nc{\dash}{{\textnormal{-}}}
\nc{\InftyCat}{\infty\dash\mrm{Cat}}
\nc{\Pres}{\mrm{Pres}}
\nc{\form}{\widehat}
\nc{\R}{\bR}
\nc{\otimesL}{\mathchoice{\overset{\bL}{\otimes}}{\otimes^\bL}{\otimes^\bL}{\otimes^\bL}}
\nc{\fibprodR}{\fibprod^\bR}
\nc{\uRHom}{\bR\uHom}
\nc{\GL}{\mrm{GL}}
\nc{\SW}{\on{SW}}
\nc{\Vect}{\on{Vect}}
\nc{\Fun}{\on{Fun}}
\nc{\Nat}{\on{Nat}}
\nc{\un}{\mbf{1}}
\nc{\pr}{\mrm{pr}}
\nc{\pt}{\mrm{pt}}
\nc{\vb}[1]{\langle{#1}\rangle}
\nc{\Pt}{\on{Pt}}
\nc{\lisse}{{\triangleleft}}
\nc{\Lis}{\mrm{Lis}}
\nc{\LisStk}{\mrm{LisStk}}
\nc{\Et}{{\mrm{Et}}}
\nc{\aff}{\mrm{aff}}
\nc{\qproj}{\mrm{qproj}}
\nc{\fp}{\mrm{fp}}
\nc{\ft}{\mrm{ft}}
\nc{\affft}{\mrm{affft}}
\nc{\sm}{\mrm{sm}}
\nc{\lci}{\mrm{lci}}
\nc{\Lisftsm}{\Lis^{\ft:\sm}}
\nc{\Lisaffftsm}{\Lis^{\affft:\sm}}
\nc{\Lisaspftsm}{\Lis^{\mrm{aspft}:\sm}}
\renc{\top}{\mrm{top}}
\nc{\C}{\on{C}}
\nc{\Chom}{\mrm{C}_\bullet}
\nc{\Ccoh}{\mrm{C}^\bullet}
\nc{\Ccohc}{\mrm{C}_{\mrm{c}}^\bullet}
\nc{\CBM}{\mrm{C}^{\BM}_\bullet}
\nc{\mot}{\mrm{mot}}
\nc{\Chommot}{\mrm{C}^{\mot}_\bullet}
\nc{\Top}{\mrm{Top}}
\renc{\top}{\mrm{top}}
\nc{\Spc}{\mrm{Spc}}
\nc{\Stk}{\mrm{Stk}}
\nc{\Art}{\mrm{Art}}
\nc{\Shv}{\on{Shv}}
\nc{\Spt}{\mrm{Spt}}
\nc{\heart}{\heartsuit}
\nc{\an}{\mrm{an}}
\nc{\Anima}{\mbf{H}}
\nc{\Aff}{\mrm{Aff}}
\nc{\MotSpc}{{\mrm{MAni}}}
\nc{\SH}{\on{\mathbf{SH}}}
\renc{\L}{\mrm{\bL}}
\nc{\h}{\mrm{h}}
\nc{\Sm}{\mrm{Sm}}
\nc{\Sch}{\mrm{Sch}}
\nc{\Asp}{\mrm{Asp}}
\nc{\Betti}{\mrm{Bet}}
\nc{\cdh}{\mrm{cdh}}
\renc{\Re}{\mrm{Re}}
\nc{\bz}{\mathbf{z}}
\nc{\Tot}{\on{Tot}}
\nc{\MGL}{\mrm{MGL}}
\nc{\inv}[1]{[\tfrac{1}{#1}]}

\nc{\scr}{\term{derived commutative ring}}
\nc{\scrs}{\term{derived commutative rings}}

\nc{\inftyCat}{\term{$\infty$-category}}
\nc{\inftyCats}{\term{$\infty$-categories}}

\nc{\inftyGrpd}{\term{$\infty$-groupoid}}
\nc{\inftyGrpds}{\term{$\infty$-groupoids}}

\nc{\dA}{\term{derived Artin}}

\title{Equivariant generalized cohomology via~stacks\vspace{-2mm}}
\author[A.\,A. Khan]{Adeel A. Khan}
\author[C. Ravi]{Charanya Ravi}
\date{2025-09-25}

\makeatletter
\def\l@subsection{\@tocline{2}{0pt}{4pc}{6pc}{}}
\makeatother

\begin{document}

\begin{abstract}
  We prove a general form of the statement that the cohomology of a quotient stack can be computed by the Borel construction.
  It also applies to the lisse extensions of generalized cohomology theories like motivic cohomology and algebraic cobordism.
  We use this to prove a (higher) equivariant Grothendieck--Riemann--Roch theorem, comparing Borel-equivariant G-theory and equivariant Chow groups.
  Finally, we give a Bernstein--Lunts-type gluing description of the \inftyCat of equivariant sheaves on a scheme $X$, in terms of nonequivariant sheaves on $X$ and sheaves on its Borel construction.
  \vspace{-5mm}
\end{abstract}

\maketitle

\renewcommand\contentsname{\vspace{-1cm}}
\tableofcontents

\setlength{\parindent}{0em}
\parskip 0.6em

\thispagestyle{empty}


\changelocaltocdepth{1}


\section*{Introduction}

  The formalism of Grothendieck's six operations on (derived categories of) étale sheaves can be extended to algebraic stacks (see \cite{LiuZheng,LaszloOlssonI}).
  Specialized to quotient stacks, this affords simple definitions of equivariant (co)homology.
  For example, if $X$ is a variety with an action of an algebraic group $G$, we may define the $G$-equivariant Borel--Moore homology of $X$ (with coefficients in a commutative ring $\Lambda$) as the hypercohomology of the complex $f^!(\Lambda_{BG})$, where $\Lambda_{BG}$ is the constant sheaf on the classifying stack $BG$ and $f : [X/G] \to BG$ is the projection from the quotient stack.

  Classically, equivariant cohomology and (Borel--Moore) homology are defined via algebraic analogues of the Borel construction \cite{BorelSeminar}.
  It is well-known that the two approaches should agree.
  In this paper, we prove a very general form of such a comparison that also applies to generalized motivic cohomology theories, and study the consequences for equivariant intersection theory.

  On first pass, let us formulate the result for Betti or $\ell$-adic sheaves.
  Given a locally of finite type Artin stack $X$ over a field $k$, let $\D(X)$ denote the \inftyCat of Betti or $\ell$-adic sheaves on $X$.
  If $G$ is an affine algebraic group over $k$ acting on $X$, let $\H^*_G(X)$ and $\H^{\BM,G}_{-*}(X)$ denote the hypercohomology groups of $f_*f^*(\Lambda_{BG})$ and $f_*f^!(\Lambda_{BG})$, respectively (where $\Lambda$ denotes the constant sheaf $\bQ$ or $\bQ_\ell$, in the Betti or $\ell$-adic cases respectively).
  We regard the Borel construction as a $G$-equivariant ind-scheme $U_\infty = \{U_\nu\}_\nu$, e.g. for $G=\GL_n$ this is the infinite Grassmannian $\mrm{Gr}_{n,\infty}$ of rank $n$ subspaces (see \secref{sec:borel} for details).
  Then we claim (see Corollaries~\ref{cor:lumberjack} and \ref{cor:tuberculation}):

  \begin{thmX}\label{thmX:homology}
    Suppose $G$ acts on a finite-dimensional Artin stack $X$ of finite type over $k$.
    Then for every integer $n\in\Z$ there are canonical isomorphisms
    \begin{align*}
      \H^n_G(X)
      &\simeq \H^n(X \fibprod^G U_\infty),\\
      \H_n^{\BM,G}(X)
      &\simeq \H^\BM_{n}(X \fibprod^G U_\infty).
    \end{align*}
  \end{thmX}

  Here we have written $X \fibprod^G U_\infty$ for the quotient ind-stack $[(X\fibprod U_\infty)/G]$, and we have (essentially by definition)
  \begin{align*}
    \H^n(X \fibprod^G U_\infty)
    &\simeq \lim_\nu \H^n(X \fibprod^G U_\nu),\\
    \H^\BM_{n}(X \fibprod^G U_\infty)
    &\simeq \lim_\nu \H^\BM_{n+2d_\nu}(X \fibprod^G U_\nu)(-d_\nu),
  \end{align*}
  where $d_\nu=\dim(U_\nu/G)$.

  For $X$ a quasi-projective $k$-scheme on which $G$ acts linearly, each $X \fibprod^G U_\nu$ is a scheme because $G$ acts freely on $U_\nu$.
  The right-hand sides of \thmref{thmX:homology} are often taken as \emph{definitions} of equivariant cohomology and Borel--Moore homology (see e.g. \cite[\S 1]{Lusztig}).
  They have also been considered in the case where $X$ itself is a stack, e.g. the moduli stack of objects in an abelian or dg-category (see \cite[\S 2.3]{Joyce}).

  Our second main result provides a complete description of the \inftyCat of sheaves on a quotient stack.
  It relates $\D([X/G])$ to the equivariant derived category of Bernstein and Lunts (compare \cite[Def.~2.1.3]{BernsteinLunts}).

  \begin{thmX}\label{thmX:cat}
    For every Artin stack $X$ locally of finite type over $k$ with $G$-action, there is a cartesian square of \inftyCats
    \[\begin{tikzcd}
      \D([X/G]) \ar{r}\ar{d}
      & \D(X \fibprod^G U_\infty) \ar{d}
      \\
      \D(X) \ar{r}
      & \D(X \times U_\infty).
    \end{tikzcd}\]
  \end{thmX}

  The \inftyCat $\D(X \times^G U_\infty)$ is the limit over $\nu$ of $\D(X \times^G U_\nu)$.
  Thus \thmref{thmX:cat} asserts in particular that a sheaf on the quotient stack $[X/G]$ amounts to the data of:
  \begin{defnlist}
    \item a sheaf $\sG \in \D(X)$;
    \item\label{item:halyard} a collection of sheaves $\sF_\nu \in \D(X \fibprod^G U_\nu)$ for every $\nu$, with compatibility isomorphisms $\sF_\nu|_{X \fibprod^G U_{\nu+1}} \simeq \sF_{\nu+1}$;
    \item\label{item:forestick} for every $\nu$, an isomorphism $\sG|_{X \times U_\nu} \simeq \sF_\nu|_{X \times U_\nu}$;
  \end{defnlist}
  and compatibilities between the isomorphisms of \itemref{item:halyard} and \itemref{item:forestick}.
  See \thmref{thm:soteriologic}.
  Note that, in contrast with \cite{BernsteinLunts}, we work with unbounded derived categories.

  To prove \thmref{thmX:homology}, we will first consider the analogous statement at the level of derived global sections.
  We will see (\thmref{thm:gawkiness} and \corref{cor:limoid}) that for any sheaf $\sF \in \D([X/G])$, the canonical map
  \begin{equation*}
    \RGamma([X/G], \sF) \to \lim_\nu \RGamma(X \fibprod^G U_\nu, \sF)
  \end{equation*}
  is invertible, where the limit is a \emph{homotopy} limit in the \inftyCat of spectra.
  The obstruction to passing from this statement to the analogous statement on hypercohomologies,
  \begin{equation*}
    \H^i([X/G], \sF) \to \lim_\nu \H^i(X \fibprod_S^G U_\nu, \sF),
  \end{equation*}
  is the vanishing of Milnor's $\lim^1$ term.
  We will show (\corref{cor:outcourt}) that this obstruction vanishes when $\sF$ is eventually coconnective with respect to the cohomological t-structure (i.e., cohomologically bounded below).

  As alluded to earlier, the main result of this paper is in fact an extension of \thmref{thmX:homology} to generalized cohomology theories such as motivic cohomology ($\approx$ higher Chow groups), algebraic cobordism, and (variants of) algebraic K-theory.
  Similarly, we will prove a version of \thmref{thmX:cat} for the stable motivic homotopy category.
  In particular, we will see:
  \begin{thmX}\label{thmX:mot}
    Let $E \in \SH(k)_{<\infty}$ be a motivic spectrum which is eventually coconnective for the homotopy t-structure.
    If $G$ acts on an Artin stack $X$ of finite type over $k$ with separated diagonal, there are canonical isomorphisms
    \[ \H^i([X/G], E)\inv{e} \simeq \lim_\nu \H^i(X \fibprod^G_S U_\nu, E)\inv{e} \]
    for every $i\in\Z$, where $e$ is the characteristic exponent of the field $k$.
  \end{thmX}

  The work of Voevodsky, Ayoub, and Cisinski--D\'eglise provides the formalism of six operations on the \inftyCats $\SH(S)$ of motivic spectra over schemes $S$ (see \cite{Ayoub,CisinskiDegliseBook}).
  Recently, this has been extended\footnote{%
    We are referring here to the \emph{lisse-extended} variant, as opposed to the genuine theory constructed in \cite{sixstack}.
    For quotient stacks this corresponds to the difference between Borel-equivariant and genuine-equivariant cohomology theories.
    An example of the latter is algebraic K-theory, as discussed in \secref{sec:k}.
  } to algebraic stacks by various authors (see \cite[\S 2]{RicharzScholbach}, \cite[App.~A]{KhanVirtual}, \cite{ChowdhuryDAngelo}, and \cite[\S 7]{Weavelisse}).
  We may thus use the above approach with quotient stacks to define equivariant Borel--Moore homology with coefficients in any motivic spectrum $E \in \SH(k)$:
  \[
    \H^{\BM,G}_s(X; E)(-r) := \pi_s \RGamma([X/G], f^!(E|_{BG})(-r)),
  \]
  where $X$ is an algebraic space\footnote{%
    Note that unlike in the Betti/étale cases, we avoid speaking of equivariant Borel--Moore homology of general Artin stacks.
    That is because in the context of $\SH$, no full account of the construction of the $!$-functors for non-representable morphisms exists in the literature at the time of writing (although see \cite[\S 4]{Weaves} for a sketch).
  } with $G$-action and $r,s\in\Z$.
  On the other hand, previous definitions via the Borel construction have long been considered in the literature already, notably by B.~Totaro \cite{TotaroChow}, Edidin--Graham \cite{EdidinGraham}, Deshpande \cite{Deshpande}, Heller--Malag\'on-L\'opez \cite{HellerMalagonLopez}, and A.~Krishna \cite{KrishnaEquivariantCobordism}.
  We will show that these approaches coincide at least \emph{at the spectrum level} (see \corref{cor:Orchestia}).
  As stated above in \thmref{thmX:mot}, we also have a comparison at the level of homotopy groups when $E$ is eventually coconnective.

  For example, in the case of the motivic cohomology spectrum, this yields a comparison of equivariant motivic Borel--Moore homology with the equivariant higher Chow groups of Edidin--Graham (\corref{cor:lactoscope}):
  \begin{equation}
    \H^{\BM,G}_{s+2n}(X; \Lambda^\mot)(-n)
    \simeq \on{A}^G_{n}(X, s) \otimes \Lambda
  \end{equation}
  for all $n,s\in\bZ$, where on the right-hand side are the $G$-equivariant higher Chow groups of $X$ \cite[\S 2.7]{EdidinGraham}.

  For $E = \KGL$ the algebraic K-theory spectrum in $\SH(k)$, we get the spectrum-level computation of lisse-extended G-theory of a quotient stack (see \corref{cor:kerflummox}):
  \begin{equation}\label{eq:retia}
    \G^\lisse([X/G]) \simeq \lim_\nu \G(X \fibprod^G U_\nu).
  \end{equation}
  Of course, $\G^\lisse(-)$ does \emph{not} agree with the genuine extension of G-theory to stacks.
  In fact, the right-hand side can often be identified with the derived completion of the G-theory spectrum $\G([X/G])$ with respect to the augmention ideal in the representation ring of $G$, see \cite{TabuadaVdB,CarlssonJoshua}.
  The formula \eqref{eq:retia} means that
  \begin{equation}
    \G^{G,\lisse}(-) := \G^\lisse([-/G]),
    \qquad
    \G_s^{G,\lisse}(-) := \pi_s \G^\lisse([-/G]), ~s\in\Z,
  \end{equation}
  is Borel-type $G$-equivariant G-theory, and on algebraic stacks $\G^\lisse(-)$ may be regarded as a ``globalization'' thereof.\footnote{%
    Since $\KGL$ is not eventually coconnective (so that \thmref{thmX:mot} does not apply), there is a potential $\lim{}\!^1$ obstruction to computing $\G^{G,\lisse}_s(X)$ as $\lim_\nu G_s(X \fibprod^G U_\nu)$.
  }

  Upon rationalization we deduce the following (higher) equivariant Grothendieck--Riemann--Roch theorem (where $(-)_\Q := (-) \otimes \Q$).

  \begin{thmX}[Equivariant GRR]\label{thmX:grr}
    Let $k$ be a field, $G$ an affine algebraic group over $k$, and $X$ a quasi-separated algebraic space of finite type over $k$ with $G$-action.
    Then for every integer $s\in\Z$ there is a canonical isomorphism
    \[
      \G^{G,\lisse}_s(X)_\Q
      \simeq \prod_{n\in\Z} \on{A}^G_{n}(X, s)_\Q,
    \]
    where on the right-hand side are the $G$-equivariant higher Chow groups of $X$ \cite[\S 2.7]{EdidinGraham}.
    Moreover, this is compatible with equivariant proper push-forwards and equivariant quasi-smooth Gysin pull-backs.
  \end{thmX}

  For $s=0$ the left-hand side becomes the completion of the (genuine) $G$-equivariant $\G_0$ of $X$ at the augmentation ideal $I_G$, so we get
  \begin{equation}
    \G_0^G(X)^\wedge_{\Q,I_G}
    \simeq \prod_{n\in\Z} \on{A}^G_{n}(X)_\Q.
  \end{equation}
  This recovers the equivariant Grothendieck--Riemann--Roch theorem of Edidin--Graham \cite[Thm.~4.1]{EdidinGrahamGRR}.
  A higher generalization of the latter was previously obtained by A.~Krishna for $X$ \emph{smooth} and quasi-projective \cite[Thm.~1.2]{KrishnaRR}.

  Consider now the case of Voevodsky's algebraic cobordism spectrum $E = \MGL \in \SH(k)$.
  For simplicity, we write for all $n,s\in\Z$
  \begin{equation*}
    \MGL^G_n(X) := \H^{\BM,G}_{2n}(X; \MGL)(-n)
  \end{equation*}
  and similarly for $\MGL_n(X)$.
  When $k$ is of characteristic zero, a geometric model $\Omega_*(-)$ for ``lower'' algebraic bordism of quasi-projective $k$-schemes was given by Levine and Morel \cite{LevineMorel}; that it agrees with $\MGL_n(X)$ was proven by Levine \cite{LevineMGL}.
  For the (lower) $G$-equivariant algebraic bordism theory
  \begin{equation*}
    \Omega^{G,\mrm{HML}}_n(X) := \lim_\nu \Omega_{n+d_\nu-g}(X\fibprod^G U_\nu),
  \end{equation*}
  considered by J.~Heller and J.~Malag\'on-L\'opez \cite{HellerMalagonLopez} (where $d_\nu=\dim(U_\nu)$ and $g=\dim(G)$), we get surjections
  \begin{equation*}
    \MGL^G_n(X) \twoheadrightarrow \Omega^{G,\mrm{HML}}_n(X)
  \end{equation*}
  from $G$-equivariant motivic Borel--Moore homology with coefficients in $\MGL$.
  But as \thmref{thmX:mot} does not apply, because $\MGL$ is not eventually coconnective, we do not know whether this map is bijective in general.
  We will see that it is so with rational coefficients (see \thmref{thm:mglq}).

  The issue seems related to the question, still open, of right-exactness of the localization sequence
  \begin{equation}\label{eq:omegloc}
    \Omega^{G,\mrm{HML}}_n(Z)
    \to \Omega^{G,\mrm{HML}}_n(X)
    \to \Omega^{G,\mrm{HML}}_n(U)
    \to 0
  \end{equation}
  for $G$-invariant closed subschemes $Z \sub X$ with open complement $U = X\setminus Z$.
  In equivariant Chow homology the analogous property is obvious (see \cite[Prop.~5]{EdidinGraham}) since any homogeneous component of $\on{A}^G_*(X)$ can be computed using a single approximation $U_\nu/G$ for large enough $\nu$ (as opposed to the entire tower $\{U_\nu/G\}_\nu$).
  In the case of $\Omega^G_*(-)$, A.~Krishna\footnote{%
    Krishna \cite{KrishnaEquivariantCobordism} used a slightly different definition of $\Omega^G_*(-)$, based on \cite{Deshpande}.
    It is shown however in \cite[Rem.~14]{HellerMalagonLopez} that the two are isomorphic.
  } showed exactness at the end (i.e., surjectivity of restriction to an open) and explained why he did not believe exactness in the middle should hold (see Prop.~5.3 in \cite{KrishnaEquivariantCobordism} and the discussion just before).
  While Heller--Malag\'on-L\'opez claimed that right-exact localization holds (see \cite[Thm.~20]{HellerMalagonLopez}), their argument does not in fact prove exactness in the middle (this gap is well-known in the area).
  We will show that $\MGL^G_*(-)$ does satisfy right-exact localization (see \thmref{thm:unchanted}), and so therefore does $\Omega^{G,\mrm{HML}}_*(-)_\Q$:

  \begin{thmX}\label{thmX:bordism}
    Let $k$ be a field of characteristic zero, $G$ an affine algebraic group over $k$, and $X$ a quasi-projective $k$-scheme with linearized $G$-action.
    Then for every $G$-invariant closed subscheme $Z \sub X$ with open complement $U = X\setminus Z$ and every $n\in\Z$, there are exact sequences
    \begin{equation*}
      \MGL^G_n(Z)
      \to \MGL^G_n(X)
      \to \MGL^G_n(U)
      \to 0
    \end{equation*}
    and
    \begin{equation*}
      \Omega^{G,\mrm{HML}}_{n}(Z)_\Q
      \to \Omega^{G,\mrm{HML}}_{n}(X)_\Q
      \to \Omega^{G,\mrm{HML}}_{n}(U)_\Q
      \to 0.
    \end{equation*}
  \end{thmX}

  \ssec{Contents}

    We begin in \secref{sec:lisse} by recalling the lisse extension construction from \cite[\S 12]{sixstack}.
    Our first key result (\propref{prop:moneyed}) is proven in \secref{sec:approx}: given a lisse-extended $\A^1$-invariant Nisnevich sheaf $F$, a stack $\sX$, and an ind-stack $\sU_\infty$ over $\sX$ satisfying appropriate hypotheses, it states that $F(\sX)$ can be computed as the limit of the pro-system $F(\sU_\infty)$.
    This is inspired by arguments extracted from \cite{MorelVoevodsky}.

    In \secref{sec:borel} we specialize to $\sX = [X/G]$ and $\sU_\infty = \{X \fibprod^G U_\nu\}_\nu$.
    \propref{prop:moneyed} implies in this case that for a lisse-extended $\A^1$-invariant sheaf of spectra $F$, its value on a quotient stack $[X/G]$ can be computed as the homotopy limit
    \begin{equation*}
      F([X/G]) \simeq \lim_\nu F(X \fibprod^G U_\nu).
    \end{equation*}
    
    Our next main technical results, contained in \secref{sec:acyclic}, deal with the question of when we can pass to homotopy groups to get isomorphisms
    \[ \pi_i F([X/G]) \simeq \lim_\nu \pi_i F(X \fibprod^G U_\nu). \]
    We will show that this is possible when $F$ is the cohomology theory represented by a sheaf in an appropriate six functor formalism, using connectivity estimates to establish the Mittag--Leffler condition; see \corref{cor:outcourt}.\footnote{%
      In a previous version, we had claimed the same for arbitrary $F$.
      However, this was based on a mistake in the formulation of \propref{prop:moneyed}.
    }

    Finally, we will apply this to Betti/étale cohomology \secref{sec:betet} (proving \thmref{thmX:homology}) and to generalized cohomology theories, i.e., cohomology theories representable in the stable motivic homotopy category, in \secref{sec:sh} (proving \thmref{thmX:mot}).
    This is applied in \secref{sec:k} to deduce equivariant Grothendieck--Riemann--Roch (\thmref{thmX:grr}).
    The implications for equivariant algebraic (co)bordism are described in \secref{sec:bordism}, which in particular contains the proof of \thmref{thmX:bordism}.

    Finally, \secref{sec:cat} is dedicated to the proof of \thmref{thmX:cat}.

  \ssec{Conventions on stacks}

    \sssec{Stacks}

      We work with \emph{higher} stacks throughout the paper.
      Thus a \emph{stack} is a presheaf of \inftyGrpds on the site of $k$-schemes that satisfies hyperdescent with respect to the étale topology.

      A morphism $f : X \to Y$ is \emph{schematic} if for every scheme $V$ and every morphism $V \to Y$, the fibred product $X \fibprod_Y V$ is a scheme.
      A stack $X$ is \emph{$0$-Artin} if it has schematic $(-1)$-truncated diagonal and there exists a scheme $U$ and an étale surjection $U \to X$.

      For $n>0$, a morphism $f : X \to Y$ is \emph{$(n-1)$-representable} if for every scheme $V$ and every morphism $V \to Y$, the fibred product $X \fibprod_Y V$ is $(n-1)$-Artin.
      A stack $X$ is \emph{$n$-Artin} if it has $(n-1)$-representable diagonal and there exists a scheme $U$ and a smooth surjection $U \to X$.

      A stack is \emph{Artin} if it is $n$-Artin for some $n$.
      
      The $1$-Artin stacks are Artin stacks (or algebraic stacks) as defined e.g. in \cite[Tag~026O]{Stacks}.
      More generally, the $n$-Artin stacks, resp. Artin stacks, are the $n$-Artin stacks, resp. higher Artin stacks, as defined in \cite[\S 4.2]{GaitsgoryStacks}.

    \sssec{The Nisnevich topology}

      Recall that the Nisnevich topology on schemes is generated by families of étale morphisms $(Y_\alpha \to X)_\alpha$ such that the morphism $\coprod_\alpha Y_\alpha \to X$ is surjective on field-valued points (see e.g. \cite[App.~A]{BachmannHoyois}).

      A smooth morphism of schemes admits étale-local sections if and only if it is surjective (see \cite[IV\textsubscript{4}, Cor.~17.16.3(ii)]{EGA}).
      Here is the Nisnevich analogue:

      \begin{lem}\label{lem:acrimonious}
        Let $f : X \to Y$ be a smooth morphism of schemes.
        Then the following conditions are equivalent:
        \begin{thmlist}
          \item\label{item:tusker}
            The morphism $f$ is surjective on field-valued points.
            
            \item\label{item:calathiscus}
            There exists a Nisnevich cover $Y' \twoheadrightarrow Y$ such that the base change $X \fibprod_Y Y' \to Y'$ admits a section.
          \end{thmlist}
          Moreover, if $Y$ is quasi-compact and quasi-separated, then $Y'$ in \itemref{item:calathiscus} can also be taken to be affine.
      \end{lem}
      \begin{proof}
        \emph{\itemref{item:calathiscus} $\implies$ \itemref{item:tusker}:}
        We will show that for every field-valued point $y : \Spec(\kappa) \to Y$, the base change $X_{\{y\}} = X \fibprod_Y \Spec(\kappa) \to \Spec(\kappa)$ admits a section.
        By base change, the condition implies that there is a Nisnevich cover $S \twoheadrightarrow \Spec(\kappa)$ such that $X_S \simeq X_{\{y\}} \fibprod_{\Spec(\kappa)} S \to S$ admits a section.
        Since $S \twoheadrightarrow \Spec(\kappa)$ is surjective on field-valued points by definition, it admits a section.
        The composition of the section $\Spec(\kappa) \to S$, the section $S \to X_S$, and the morphism $X_S \to X_\kappa$ is then a section of $X_{\{y\}} \to \Spec(\kappa)$ as desired.

        \emph{\itemref{item:tusker} $\implies$ \itemref{item:calathiscus}:}
        Let $y : \Spec(\kappa) \to Y$ be a point and $x : \Spec(\kappa) \to X$ a lift.
        Since $f : X \to Y$ is smooth, there exists by \cite[IV\textsubscript{4}, 18.6.6(i), 18.5.17]{EGA} a morphism $\tilde{x} : S \to X$ extending $x$, where $S$ is the henselization of $Y$ at $y$.
        Recall that $S$ can be identified with the cofiltered limit of elementary étale neighbourhoods\footnote{%
          Here an elementary étale neighbourhood is an étale morphism $Y' \to Y$ along which $y$ lifts to $y' : \Spec(\kappa) \to Y'$.
        } of $(Y,y)$.
        Since $X$ is locally of finite presentation over $Y$, it follows that there exists an étale neighbourhood $Y'_y \to Y$ over $y$ such that the $Y$-morphism $\tilde{x} : S \to X$ factors through $Y'_y \to X$.
        Then the disjoint union $Y' = \coprod_y Y'_y \twoheadrightarrow Y$ over all field-valued point $y$ is an étale morphism which is surjective on field-valued points, i.e., a Nisnevich cover, with the desired property.
        If $Y$ is quasi-compact, then there is a finite subcover refining $Y'$ (see e.g. \cite[Lem.~2.1.2]{HoyoisMilnor}), so in particular we may take $Y'$ quasi-compact.
        We may then further replace $Y'$ by a Zariski cover by an affine scheme.
      \end{proof}

    \sssec{Fibre dimension}
    \label{sssec:fibdim}

      We recall the notion of fibre dimension for a morphism of stacks, as studied in \cite[IV\textsubscript{3}, \S 13]{EGA} and \cite[Tag~0DRQ]{Stacks}.

      Let $f : \sX \to \sY$ be a locally of finite type $1$-representable morphism of locally noetherian Artin stacks.
      Given a field $\kappa$ and a $\kappa$-point $x$ of $\sX$, the \emph{fibre dimension} of $f$ at $x$ is the dimension at $x$ of the topological space underlying the fibre $f^{-1}(f(x)) = \sX_{f(x)} = \sX \fibprod_\sY \Spec(\kappa)$.
      See \cite[Tag~0DRQ]{Stacks} for details, and \cite[Tag~04XG]{Stacks} for the definition of the underlying topological space of a $1$-Artin stack.

      The fibre dimension of $f : \sX \to \sY$ determines a function on the underlying topological space of $\sX$.
      This function is upper semi-continuous; see \cite[IV\textsubscript{3}, Thm.~13.1.3]{EGA} and \cite[Tag~0DRE]{Stacks}.

      \begin{lem}\label{lem:dimcodim}
        Let $f : Y \to X$ be a smooth morphism of locally noetherian schemes and $i : Z \hook Y$ a closed immersion.
        Let $d_{Y/X}$ and $d_{Z/X}$ denote the fibre dimensions of $f$ and $g = f\circ i$.
        For every point $z$ of $Z$, we have $d_{Y/X}(z) - d_{Z/X}(z) = \codim(Z\fibprod_X \{g(z)\}, Y\fibprod_X \{g(z)\})$.
      \end{lem}
      \begin{proof}
        Let $z : \Spec(k(z)) \to Z$ be a point of $Z$, with image $x = g(z)$ in $X$.
        By definition, the fibre dimensions of $f$ and $g$ at $z$ are
        \[
          d_{Y/X}(z) = \dim_z(Y \fibprod_X \{x\}),
          \quad d_{Z/X}(z) = \dim_z(Z \fibprod_X \{x\}),
        \]
        respectively.
        Since $Y$ is smooth over $X$, $Y \fibprod_X \{x\}$ is in particular biequidimensional (combine \cite[0\textsubscript{IV}, Cor.~16.5.12]{EGA}, \cite[IV\textsubscript{2}, Cor.~5.10.9]{EGA}, and \cite[Lem.~2.6]{Heinrich}).
        It follows then from \cite[Prop.~4.1]{Heinrich} that
        \[ d_{Y/X}(z) - d_{Z/X}(z) = \codim(Z \fibprod_X \{x\}, Y \fibprod_X \{x\}) \]
        as claimed.
      \end{proof}

  \ssec{Notation}

    We work over a fixed commutative ring $k$, which we leave implicit in the notation.
    We denote by $\Sch$ (resp. $\Asp$) the category of schemes (resp. quasi-separated algebraic spaces) of finite type over $k$.

    We denote by $\Stk$ the \inftyCat of Artin stacks locally of finite type over $k$.
    Given $S\in\Stk$ we denote by $\Stk_S$ the \inftyCat of Artin stacks locally of finite type over $S$.
    Note that any morphism in $\Stk$ (resp. $\Stk_S$) is automatically locally of finite type.

  \ssec{Acknowledgments}

    We would like to thank Tom Bachmann, Jeremiah Heller, Jens Hornbostel, Henry July, Marc Levine, David Rydh, and Pavel Safronov for comments, questions, and discussions.
    We are especially grateful to Tom Bachmann and Marc Levine for pointing out several errors in previous revisions.
  
    The first-named author spoke about a preliminary version of these results at the INI workshop ``Algebraic K-theory, motivic cohomology and motivic homotopy theory'' in June 2022.

    We acknowledge support from MOST grant 110-2115-M-001-016-MY3 (A.K.) and EPSRC grant no EP/R014604/1 (A.K. and C.R.).
    We would like to thank the Isaac Newton Institute for Mathematical Sciences, Cambridge, for support and hospitality during the programme KAH2 where work on this paper was undertaken.


\section{Lisse extension}
\label{sec:lisse}

  Given $\sX\in \Stk$, we denote by $\Lis_\sX$ (resp. $\LisStk_\sX$) the \inftyCat of pairs $(T, t)$ where $T \in \Sch$ (resp. $T \in \Stk$) and $t : T \to \sX$ is a smooth morphism\footnote{%
    Since $t : T \to \sX$ is schematic for $T \in \Sch$, it is smooth if and only if for any morphism $T' \to \sX$ where $T'$ is a scheme, the morphism of schemes $T \fibprod_\sX T' \to T'$ is smooth.
    See e.g. \cite[Chap.~2, \S 4.1]{GaitsgoryRozenblyum}.
  }.
  Let $F : \Lis_\sX^\op \to \sV$ be a presheaf, where $\sV$ is an \inftyCat with limits.

  \begin{defn}\label{defn:intraoctave}
    The \emph{lisse extension} of $F$ is the presheaf $$F^\lisse : \LisStk_\sX^{\op} \to \sV$$ defined as the right Kan extension of $F$ along the fully faithful functor $\Lis_\sX \hook \LisStk_\sX$.
    In particular, we have
    \[
      F^\lisse(\sX) \simeq \lim_{(T,t)} F(T)
    \]
    where the limit is taken over $(T,t)\in\Lis_\sX$.
  \end{defn}

  Given a presheaf $F : \Sch^\op \to \sV$, we may restrict along the forgetful functor $\Lis_\sX \to \Sch$ and form the lisse extension of the resulting presheaf $F_\sX$ on $\Lis_\sX$.
  On the other hand, we may also describe $F_\sX^\lisse$ more directly in terms of $F$.
  Moreover, this will work for presheaves with restricted functoriality (e.g. only for smooth or lci morphisms).

  \begin{defn}\label{defn:uncircumspectly}
    Let $\Stk^? \sub \Stk$ be a wide subcategory containing all smooth morphisms, and denote by $\Sch^? \sub \Sch$ the intersection $\Sch \cap \Stk^?$.
    Given a presheaf $F : \Sch^{?,\op} \to \sV$, its \emph{lisse extension} $F^\lisse$ is its right Kan extension
    $$F^\lisse : \Stk^{?,\op} \to \sV$$
    along $\Sch^{?} \hook \Stk^{?}$.
  \end{defn}

  \begin{rem}
    Let $\Sch^?_{/\sX}$ denote the full subcategory of the slice $\Stk^?_{/\sX}$ spanned by pairs $(T, t : T \to \sX)$ where $T$ is a scheme (and $t$ is a morphism in $\Stk^?$).
    Note that morphisms $(T,t) \to (T',t')$ are morphisms $T \to T'$ in $\Sch^?$ which are compatible with $t$ and $t'$.
    Then we have
    \[
      F^\lisse(\sX) \simeq \lim_{(T,t) \in \Sch^?_{/\sX}} F(T).
    \]
  \end{rem}

  \begin{prop}\label{prop:entropionize}
    Let $F : \Sch^{?,\op} \to \sV$ be a presheaf.
    Assume that $\Sch^?$ contains all lci morphisms and that its morphisms are stable under smooth base change\footnote{%
      i.e., if $X \to Y$ is a morphism in $\Sch^?$ and $Y' \to Y$ is a smooth morphism in $\Sch$, then the base change $X \fibprod_Y Y' \to Y'$ belongs to $\Sch^?$
    }.
    If $F$ satisfies Nisnevich descent, then for every $\sX \in \Stk$ there is a canonical isomorphism
    \[
      F^\lisse|_{\LisStk_\sX} \to (F|_{\Lis_\sX})^\lisse
    \]
    of presheaves on $\LisStk_\sX$.
  \end{prop}

  For example, we will apply \propref{prop:entropionize} to presheaves on $\Sch$ and the subcategories
  \[
    \Sch^{\sm},
    ~\Sch^{\lci}
    \quad \sub \quad \Sch
  \]
  containing only the smooth and lci morphisms, respectively.

  \begin{lem}\label{lem:coccus}
    Let $F : \Sch^{?,\op} \to \sV$ be a presheaf where $\Sch^?$ satisfies the conditions of \propref{prop:entropionize}.
    If $F$ satisfies $\tau$-descent, where $\tau$ stands for the Nisnevich or étale topology, then for every $\sX \in \Stk$, every scheme $U$ and every smooth morphism $p : U \to \sX$ admitting $\tau$-local sections, the canonical map
    \begin{equation}\label{eq:rhyacolite}
      F^\lisse(\sX) \to \Tot(F(U_\bullet))
    \end{equation}
    is invertible, where $U_\bullet$ is the \v{C}ech nerve of $p$, ``Tot'' denotes the totalization of a cosimplicial object.
  \end{lem}
  \begin{proof}
    Assume first that $X=\sX$ is a scheme.
    By assumption, there exists a scheme $V$ and a $\tau$-cover $V \twoheadrightarrow X$ over which $p$ admits a section.
    Since $F$ satisfies descent for the \v{C}ech nerve of $V\twoheadrightarrow X$, we may replace $X$ by $V$ and thereby assume that $p$ admits a section.
    This section (which is lci, hence determines a morphism in $\Sch^?$) gives rise to a splitting of the augmented simplicial object $U_\bullet \to X$, so that the map \eqref{eq:rhyacolite} is invertible by \cite[Lem.~6.1.3.16]{LurieHTT}.

    Now we consider the general case.
    For every pair $(T,t)$ where $T \in \Sch$ and $t : T \to \sX$ is a morphism in $\Stk^?$, denote by $U_T \twoheadrightarrow T$ the base change of $p$ and by $U_{T,\bullet}$ its \v{C}ech nerve.
    The canonical map
    \begin{equation}\label{eq:metaloscopy}
      F(T) \to \Tot(F(U_{T,\bullet}))
    \end{equation}
    is invertible by above.

    Note that for every smooth morphism $Y \to \sX$ from a scheme, the base change functor $\Sch^{?,\op}_{/\sX} \to \Sch^{?,\op}_{/Y}$ is cofinal.
    Indeed, given $(T,t) \in \Sch^?_{/Y}$ we have a section $s : T \to T \fibprod_\sX Y$ over $Y$ of the projection $T \fibprod_\sX Y \to T$.
    Since the latter is smooth, $s$ is lci and determines a morphism in $\Sch^?_{/Y}$ whose target lies in the essential image of the functor in question.
    Thus, the morphism \eqref{eq:rhyacolite} is the limit over $(T,t)\in\Sch^?_{/\sX}$ of the isomorphisms \eqref{eq:metaloscopy}.
  \end{proof}

  \begin{proof}[Proof of \propref{prop:entropionize}]
    It will suffice to show that for every $\sY \in \LisStk_\sX$, the projection map
    \begin{equation}\label{eq:subhealth}
      F^\lisse(\sY) \simeq \lim_{(T,t)\in\Sch^?_{/\sY}} F(T)
      \to \lim_{(T,t)\in(\Lis_\sX)_{/\sY}} F(T)
    \end{equation}
    is invertible.
    Here $(\Lis_\sX)_{/\sY}$ is the \inftyCat of pairs $(T, t)$ where $T \in \Lis_\sX$ and $t : T \to \sY$ is a morphism in $\LisStk_\sX$.

    By \cite[Thm.~5.1]{ChowdhuryDAngelo}, there exists a scheme $Y$ and a smooth morphism $p : Y \twoheadrightarrow \sY$ admitting \emph{Nisnevich}-local sections.\footnote{%
      A priori, an Artin stack $\sY$ admits by definition such a morphism admitting \emph{étale}-local sections.
    }
    Denote by $Y_\bullet$ the \v{C}ech nerve of $p$, so that there is an equivalence $\Tot(F(Y_\bullet)) \simeq F^\lisse(\sY)$ by \lemref{lem:coccus}.
    This defines a diagram $\bDelta^\op \to (\Lis_\sX)_{/\sY}$, so by projection there is a canonical map
    \[
      \lim_{(T,t)\in (\Lis_\sX)_{/\sY}} F(T)
      \to \lim_{[n]\in\bDelta} F(Y_n)
      \simeq F^\lisse(\sY).
    \]
    One verifies that this is inverse to \eqref{eq:subhealth}.
  \end{proof}


\section{A pro-approximation lemma}
\label{sec:approx}

This section contains the first key technical result of the paper (\propref{prop:moneyed}).
We begin by recalling some preliminaries about pro-objects.

\begin{rem}\label{rem:proscolex}
  Let $\cV$ be an \inftyCat with limits.
  We denote by $\Pro(\cV)$ the \inftyCat of pro-objects in $\cV$ (see e.g. \cite{BarneaHarpazHorel}).
  There is a canonical fully faithful functor $\cV \hook \Pro(\cV)$ sending an object $V \in \cV$ to the constant pro-object $\{V\}$, whose essential image generates $\Pro(\cV)$ under cofiltered limits (see \cite[Cor.~3.2.14]{BarneaHarpazHorel}).
  The functor $\cV \hook \Pro(\cV)$ preserves finite limits, which are computed levelwise in $\Pro(\cV)$; indeed, this is easily verified by universal properties using the formula for mapping \inftyGrpd
  \begin{equation}\label{eq:muconic}
    \Maps_{\Pro(\cV)}(\{X_\alpha\}_\alpha, \{Y_\beta\}_\beta) \simeq \lim_\beta \colim_\alpha \Maps_{\cV}(X_\alpha, Y_\beta).
  \end{equation}
  In particular, it follows that $\Pro(\cV)$ admits arbitrary limits by \cite[Prop.~4.4.2.6]{LurieHTT}.
  Moreover, formation of limits defines a functor
  \begin{equation}\label{eq:inexpugnable}
    \Pro(\cV) \to \cV,
    \quad \{X_\alpha\}_\alpha \mapsto \lim_\alpha X_\alpha,
  \end{equation}
  which is right adjoint to $\cV \hook \Pro(\cV)$ (again using \eqref{eq:muconic}), and in particular limit-preserving.
\end{rem}

The following definition is inspired by \cite[\S 4, Def.~2.1]{MorelVoevodsky} (see also \remref{rem:latticed} below).

\begin{notat}\label{notat:duces}
  Let $\sX \in \Stk$ and $\{\sU_\nu\}_\nu$ a filtered diagram in $\Stk_\sX$ with monomorphisms as transition maps.
  Suppose that for every index $\nu$ there is a vector bundle $\sV_\nu$ over $\sX$, an open immersion $\sU_\nu \hook \sV_\nu$ over $\sX$, and a closed substack $\sW_\nu \sub \sV_\nu$ complementary to $\sU_\nu$ containing the zero section, such that the following conditions hold:
  \begin{defnlist}
    \item\label{item:endellionite}
    For every affine scheme $T$ and every morphism $T \to \sX$, there exists an index $\nu_0$ such that the morphism $\sU_{\nu_0} \fibprod_\sX T \to T$ admits Nisnevich-local sections.

    \item\label{item:neuropathology}
    For every index $\nu$, there exists $\mu>\nu$ such that the transition map $\sU_\nu \to \sU_\mu$ factors as follows:
    \[\begin{tikzcd}
      \sV_\nu\setminus \sW_\nu \ar[hookrightarrow]{r}{(0,\id)}\ar[equals]{d}
      & \sV_\nu\fibprod_\sX \sV_\nu \setminus \sW_\nu\fibprod_\sX \sW_\nu \ar[dashed]{d}
      \\
      \sU_\nu \ar[hookrightarrow]{r}
      & \sU_\mu.
    \end{tikzcd}\]
  \end{defnlist}
\end{notat}

A presheaf $F : \Lis_\sX^\op \to \cV$ is \emph{$\A^1$-invariant} if for every $T \in \Lis_\sX$, the canonical map $F(T) \to F(T \times \A^1)$ is invertible.

\begin{prop}\label{prop:moneyed}
  Let $\sX\in\Stk$ and $\{\sU_\nu\}_\nu$ as in \notatref{notat:duces}.
  Let $F : \Lis_\sX^\op \to \cV$ be an $\A^1$-invariant Nisnevich sheaf and write $\overline{F} : \Lis_\sX^\op \to \cV \hook \Pro(\cV)$ for the composite with the canonical embedding (\remref{rem:proscolex}).
  Then there is a canonical isomorphism
  \begin{equation}\label{eq:Undine}
    \overline{F}^\lisse(\sX) \simeq \{F^\lisse(\sU_\nu)\}_\nu
  \end{equation}
  in $\Pro(\cV)$.
  In particular, the canonical morphisms $F^\lisse(\sX) \to F^\lisse(\sU_\nu)$ determine an isomorphism
  \begin{equation}\label{eq:seronegative}
    F^\lisse(\sX) \simeq \lim_\nu F^\lisse(\sU_\nu)
  \end{equation}
  in $\cV$.
\end{prop}

\begin{rem}
  Note that $\overline{F}^\lisse(\sX)$ need not be isomorphic to the constant pro-system $\{ F^\lisse(\sX) \}$.
  In particular, \propref{prop:moneyed} does not imply the existence of an isomorphism
  \begin{equation*}
    \{ F^\lisse(\sX) \}
    \simeq \{ F^\lisse(U_\nu) \}_\nu.
  \end{equation*}
  In fact, $\overline{F}^\lisse(\sX)$ may not be essentially constant at all, see \remref{rem:spearwood} for an example.
\end{rem}

The proof of \propref{prop:moneyed} will require the following lemmas.
We denote by $\Anima$ the \inftyCat of \inftyGrpds (or homotopy types).

\begin{lem}\label{lem:tawdrily}
  Let $\sX\in\Stk$ and $F : \Lis_\sX^\op \to \cV$ be a presheaf.
  Denote by
  \[
    F_! : \Fun(\Lis_\sX^\op, \Anima)
    \to \Pro(\cV)^\op
  \]
  the unique colimit-preserving functor which restricts to $\overline{F} : \Lis_\sX^\op \to \cV \hook \Pro(\cV)$ (notation as in \propref{prop:moneyed}).
  Let $A \to B$ be an \emph{$(\A^1,\Nis)$-local equivalence} in $\Fun(\Lis_\sX^\op, \Anima)$, i.e. a morphism such that the induced map
  \[ \Maps(B, G) \to \Maps(A, G) \]
  is invertible for every $\A^1$-invariant Nisnevich sheaf $G \in \Fun(\Lis_\sX^\op, \Anima)$.
  If $F$ satisfies $\A^1$-invariance and Nisnevich descent, then $F_!(A) \to F_!(B)$ is invertible.
\end{lem}
\begin{proof}
  By \cite[Prop.~5.5.4.15(4)]{LurieHTT}, a morphism $A \to B$ in $\Fun(\Lis_\sX^\op, \Anima)$ is an $(\A^1,\Nis)$-local equivalence if and only if it belongs to the strongly saturated class generated by the morphisms
  \begin{defnlist}
    \item\label{item:bitterwood}
    for every $T \in \Lis_\sX$, the projection $T \times \A^1 \to T$;
    \item\label{item:sclerotica}
    \newcounter{counter:sclerotica}
    \setcounter{counter:sclerotica}{\value{defnlisti}}
    for every Nisnevich covering family $(T'_\alpha \to T)_\alpha$ in $\Lis_\sX$, the morphism $\colim T'_\bullet \to T$ where $T'_\bullet : \bDelta^\op \to \Fun(\Lis_\sX^\op, \Anima)$ is the \v{C}ech nerve of $\coprod_\alpha T'_\alpha \to T$.
  \end{defnlist}
  By (the proof of) \cite[Thm.~2.2.7]{KhanLocalization}, \itemref{item:sclerotica} may be replaced by the following class:
  \begin{defnlistbis}[start=\value{counter:sclerotica}]
    \item\label{item:scleroticabis} the morphism from the initial presheaf to the presheaf represented by the empty scheme; and for every étale morphism $V \to T$ in $\Lis_\sX$ which is an isomorphism away from a cocompact closed subset $K \sub \abs{T}$, the morphism $V \coprod_W U \to T$ in $\Fun(\Lis_\sX^\op, \Anima)$, where $U = T\setminus K$ and $W = V \fibprod_T U$.
  \end{defnlistbis}
  Since $F_!$ preserves colimits, it will thus suffice to show that it inverts morphisms of type \itemref{item:bitterwood} and \itemref{item:scleroticabis}.
  For $T \in \Lis_\sX$, $F_!$ sends $T \times \A^1 \to T$ to the morphism of constant pro-objects
  \[ \{F(T)\} \to \{F(T\times\A^1)\} \]
  which is invertible since $F$ is $\A^1$-invariant.
  Similarly, $F_!$ sends $V \coprod_W U \to T$ as in \itemref{item:scleroticabis} to the morphism of pro-objects
  \[ \{F(T)\} \to \{F(V)\} \fibprod_{\{F(W)\}} \{F(U)\}. \]
  Since finite limits in $\Pro(\cV)$ are computed levelwise (\remref{rem:proscolex}), this is identified with the morphism of constant pro-objects $$\{F(T)\} \to \{F(V) \fibprod_{F(W)} F(U)\},$$ which is invertible since $F$ satisfies Nisnevich descent (and by \cite[Thm.~2.2.7]{KhanLocalization}).
\end{proof}

\begin{lem}\label{lem:adstipulate}
  Let $X = \sX$ be an affine scheme and $\{U_\nu\}_\nu = \{\sU_\nu\}_\nu$ be as in \notatref{notat:duces}.
  Suppose there exists an index $\nu_0$ such that $U_{\nu_0} \to X$ admits a section.
  Then the presheaf $U_\infty := \colim_\nu U_\nu$ (where the colimit is taken in presheaves) is $\A^1$-contractible on smooth affine $X$-schemes.
\end{lem}
\begin{proof}
  The following argument is extracted from the proof of \cite[\S 4, Prop.~2.3]{MorelVoevodsky}.
  The claim is that the \inftyGrpd $\RGamma(T, \on{L}_{\A^1} U_\infty)$ is contractible for every affine $T \in \Lis_X$, where $\on{L}_{\A^1}$ denotes the $\A^1$-localization functor (see e.g. \cite[Proof of Prop.~C.6]{HoyoisLefschetz}), i.e. that the simplicial set
  \[
    \Maps_{\Fun(\Lis_X^\op,\Anima)}(T \times \A^\bullet, U_\infty)
    \simeq \colim_\nu \Maps_{\Lis_X}(T \times \A^\bullet, U_\nu)
  \]
  is a contractible Kan complex.
  By \cite[Lem.~A.2.6]{EHKSY} and closed gluing for the presheaf $U_\infty$, it is enough to show that for every $n\ge 0$ and every affine $T \in \Lis_X$, the restriction map
  \[ \Maps(T\times \A^n, U_\infty) \to \Maps(T\times\partial\A^n, U_\infty) \]
  is surjective on $\pi_0$, where we identify $\A^n$ with the closed subscheme of $\A^{n+1} = \Spec(\Z[T_0,\ldots,T_n])$ defined by $\sum_i T_i = 1$, and $\partial\A^n$ is the closed subscheme defined by the further equation $T_0\dots T_n = 0$.

  Let $\nu\ge\nu_0$ be an index.
  Denote by $s : X \to U_{\nu_0} \to U_\nu$ the induced section and by $t : T \to X \to U_\nu$ its composite with the structural morphism.
  The existence of $t$ shows the surjectivity for $n=0$.

  Let $n>0$ and $f : T \times \partial\A^n \to U_\nu$ a morphism over $X$.
  We claim that this extends to a morphism $g : T \times \A^n \to U_\mu$ for some index $\mu\ge\nu$.
  Since $T$ and $X$ are affine and $V_\nu$ is a vector bundle over $X$, there exists an $X$-morphism $g' : T\times \A^n \to V_\nu$ which restricts to $f$ on $T\times\partial\A^n$.
  Since $T\times\partial\A^n$ and $g'^{-1}(W_\nu)$ are disjoint as closed subschemes of $T\times\A^n$, there exists for the same reason an $X$-morphism $g'' : T\times\A^n \to V_\nu$ which restricts to $0$ on $T\times\partial\A^n$ and to
  $$g'^{-1}(W_\nu) \to X \xrightarrow{s} U_\nu \sub V_\nu$$
  on $g'^{-1}(W_\nu)$.
  By construction, the induced $X$-morphism
  \[
    (g'',g') :
    T\times\A^n \to V_\nu\fibprod_X V_\nu
  \]
  restricts to $(0,f) : T\times\partial\A^n \to V_\nu \fibprod_X V_\nu$, and factors through the complement of $W_\nu\fibprod_X W_\nu$.
  Let $\mu>\nu$ and the morphism $V_\nu\fibprod_X V_\nu \setminus W_\nu \fibprod_X W_\nu \to U_\mu$ be as in assumption~\ref{item:neuropathology}.
  Then the composite
  \[
    g : T\times\A^n \xrightarrow{(g'',g')} V_\nu\fibprod_T V_\nu \setminus W_\nu \fibprod_X W_\nu \to U_\mu
  \]
  fits into the commutative diagram
  \[\begin{tikzcd}
    T\times\partial\A^n \ar{r}{f}\ar[hookrightarrow]{d}
    & U_\mu \ar{d}
    \\
    T\times\A^n \ar{r}\ar[dashed]{ru}{g}
    & X
  \end{tikzcd}\]
  as desired.
\end{proof}

\begin{proof}[Proof of \propref{prop:moneyed}]
  The second isomorphism \eqref{eq:seronegative} follows from \eqref{eq:Undine} by applying the limit-preserving functor $\Pro(\cV) \to \cV$ \eqref{eq:inexpugnable}.
  By \lemref{lem:tawdrily} it will suffice to show that $\colim_\nu \sU_\nu \to \sX$ is an $(\A^1,\Nis)$-local equivalence in $\Fun(\Lis_\sX^\op, \Anima)$.
  By universality of colimits, this morphism is identified with the colimit over $T \in \Lis_\sX$ of the base changes
  \[ \colim_\nu \sU_\nu \fibprod_\sX T \to T. \]
  Since local equivalences are preserved by the colimit-preserving functor $\Fun(\Lis_T^\op, \Anima) \to \Fun(\Lis_\sX^\op, \Anima)$ sending $U \in \Lis_T$ to $U \in \Lis_\sX$, and because the data and assumptions in \notatref{notat:duces} are stable under base change, we may thus replace $\sX$ by $T$ to assume that it is a scheme.
  We can moreover assume that $X := \sX$ is affine, arguing similarly using the fact that it can be written as a colimit of affines up to local equivalence.
  Finally, we may also assume up to local equivalence that $U_{\nu_0} \to X$ admits a section for some index $\nu_0$ (by condition~\itemref{item:endellionite}).
  Now the claim follows from \lemref{lem:adstipulate}.
\end{proof}

\begin{rem}\label{rem:latticed}
  In \notatref{notat:duces}, a sufficient condition for \itemref{item:endellionite} is that for every field $\kappa$ and every $\kappa$-valued point $s : \Spec(\kappa) \to \sX$, there exists an index $\nu_s$ and a lift $\Spec(\kappa) \to \sU_{\nu_s}$.
  Indeed, let us show that if $\sX$ is affine then there exists an index $\nu_0$ such that $\sU_\nu \to \sX$ admits Nisnevich-local sections.
  The assumption implies that the disjoint union $\coprod_\nu \sU_\nu \to \sX$ is a smooth morphism (not necessarily of finite type) which is surjective on field-valued points.
  By \lemref{lem:acrimonious}, there exists an affine scheme $X$ and a Nisnevich cover $X \twoheadrightarrow \sX$ along which the base change $\coprod_\nu \sU_\nu \fibprod_\sX X \to X$ admits a section.
  Since $X$ is quasi-compact, there is a finite subset $I$ of indices through which the section factors.
  Any section of $\sU_\nu$ gives rise to a section of $\sU_\mu$ for any $\mu>\nu$ (by composition with the transition map), so we may assume that $I$ consists of a single index $\nu_0$.
\end{rem}


\section{The Borel construction}
\label{sec:borel}

We now specialize our general results from the previous sections to the case of the Borel construction.
There are various interpretations of the latter in the algebraic category (compare \cite[1.1]{Lusztig}, \cite[\S 1]{TotaroChow}, \cite[\S 4.2]{MorelVoevodsky}, \cite[Def.~10]{HellerMalagonLopez}).
We will use the following variant.

\begin{notat}\label{notat:islandlike}
  Let $S$ be an algebraic space, locally of finite type over $k$, and $G$ an fppf group scheme over $S$.
  We construct an $\bN_{>0}$-indexed tower $\{U_\nu\}_\nu$ where each $U_\nu$ is an algebraic space of finite type over $S$ with a free $G$-action, and the transition maps are $G$-equivariant closed immersions.
  The construction will make use of the following auxiliary choices:
  \begin{defnlist}
    \item\label{item:islandlike/G}
    An embedding of $G$ as a closed subgroup scheme of $\GL_S(\sE)$ for some finite locally free sheaf $\sE$ on $S$.
    Write $V = \V_S(\sE)$ for the associated vector bundle over $S$ with its induced $G$-action.

    \item\label{item:islandlike/U}
    A open $U \sub V$ on which $G$ acts freely, whose (reduced) complement $W \sub V$ is of fibre dimension over $S$ \sssecref{sssec:fibdim} strictly smaller than that of $V$ over $S$.
  \end{defnlist}
  For every integer $\nu>0$, set $V_\nu := V^{\times \nu}$ and $W_\nu := W^{\times \nu}$ (where both powers are fibred over $S$), and let $U_\nu \sub V_\nu$ be the complement of $W_\nu$.
  The diagonal action of $G$ on $V^{\times \nu}$ restricts to a free action on $U_\nu$.
  The transition map $U_\nu \hook U_{\nu+1}$ is defined by restricting the $G$-equivariant closed immersion $(\id, 0) : V^{\times \nu} \hook V^{\times \nu} \fibprod_S V = V^{\times \nu+1}$.
  
  Since the quotients $[U_\nu/G]$ are algebraic spaces (as the actions are free), we will also write $U_\nu/G := [U_\nu/G]$.
\end{notat}

\begin{exam}
  Suppose $G$ is embeddable, in the sense that some embedding $G \sub \GL_S(\sE)$ as in \itemref{item:islandlike/G} exists (e.g. $S$ is the spectrum of a field and $G$ is a linear algebraic group).
  Then the choice of some $U$ as in \itemref{item:islandlike/U} also exists, see e.g. \cite[Rem.~1.4]{TotaroChow}.
\end{exam}

\begin{rem}\label{rem:borelfeatures}
  The following two features of the construction in \notatref{notat:islandlike} will be useful later:
  \begin{defnlist}
    \item
    For each $\nu$ consider the function
    \[
      c_\nu := d_{V_\nu/S}|_{W_\nu} - d_{W_\nu/S}
    \]
    where $d_{V_\nu/S}$ and $d_{W_\nu/S}$ denote the fibre dimensions of $V_\nu$ and $W_\nu$ over $S$, respectively.\footnote{%
      See \sssecref{sssec:fibdim}.
      By \lemref{lem:dimcodim}, $c_\nu$ can be computed at any point of $W_\nu$ as the fibrewise codimension.
    }
    We have $c_\nu = \nu \cdot c_1$ (this follows from \cite[IV\textsubscript{2}, Cor.~4.1.5]{EGA}), so in particular $c_\nu$ tends to $\infty$ as $\nu$ increases.

    \item
    For all indices $\mu>\nu$, the open immersion $U_\nu \fibprod_S V_{\mu-\nu} \hook V_\nu \fibprod_S V_{\mu-\nu} = V_\mu$ factors through the open $U_\mu \sub V_\mu$.
    In other words, the transition map $U_\nu \hook U_\mu$ factors as follows:
    \[\begin{tikzcd}
      U_\nu \ar[hookrightarrow]{r}{(\id,0)}\ar[hookrightarrow]{d}
      & U_\nu\fibprod_S V_{\mu-\nu} \ar[hookrightarrow]{d}\ar[dashed]{ld}
      \\
      U_\mu \ar[hookrightarrow]{r}
      & V_\mu = V_\nu \fibprod_S V_{\mu-\nu}.
    \end{tikzcd}\]
  \end{defnlist}
\end{rem}

\begin{notat}
  Let $\Stk_S^G$ denote the \inftyCat of locally of finite type Artin stacks $X$ over $S$ with $G$-action.
  This is equivalent to the \inftyCat $\Stk_{BG}$ of locally of finite type Artin stacks $\sX$ over $BG = [S/G]$ via the assignment $X \mapsto \sX = [X/G]$.
  For $X \in \Stk_S^G$, we write $$X \fibprod^G_S U_\nu := [X/G] \fibprod_{BG} (U_\nu/G)$$ for each $\nu$.
  This is representable and of finite type over $[X/G]$.
\end{notat}

\begin{rem}\label{rem:unsainted}
  If $X$ is a quasi-projective scheme over $S$ with a linearized $G$-action, then each $X \fibprod^G_S U_\nu$ is a quasi-projective $S$-scheme (see \cite[Prop.~7.1]{MumfordFogartyKirwan}).
\end{rem}

Applying the results of \secref{sec:approx}, we obtain:

\begin{thm}\label{thm:gawkiness}
  Let $X \in \Stk^G_S$ and let $F : \Lis_{[X/G]}^\op \to \cV$ be an $\A^1$-invariant Nisnevich sheaf with values in an \inftyCat $\cV$ with limits.
  Then there is a canonical isomorphism
  \begin{equation}
    F^\lisse([X/G]) \simeq \lim_\nu F^\lisse(X\fibprod_S^G U_\nu)
  \end{equation}
  in $\cV$.
\end{thm}
\begin{proof}
  The filtered diagram $\{U_\nu\}_\nu$ over $S$ satisfies the assumptions of \notatref{notat:duces} (compare \cite[\S 4, Ex.~2.2]{MorelVoevodsky} and \remref{rem:latticed}).
  Moreover, the transition maps and the maps $U_\nu \hook V_\nu$ are all $G$-equivariant, so the same holds for the quotient $\{U_\nu/G\}_\nu$ over $BG = [S/G]$.
  By base change it also holds for $\{X\fibprod^G_S U_\nu\}_\nu$ over $[X/G]$.
  Thus the claim follows from \propref{prop:moneyed}.
\end{proof}

\begin{cor}\label{cor:hemiataxy}
  Let $F : \Sch^{\lci,\op} \to \Spt$ be an $\A^1$-invariant Nisnevich sheaf of spectra.
  Then for every $X\in \Stk_S^G$, there is a canonical isomorphism of spectra
  \begin{equation*}
    F^\lisse([X/G]) \to \lim_\nu F^\lisse(X\fibprod_S^G U_\nu),
  \end{equation*}
  where $F^\lisse : \Stk^{\lci,\op} \to \Spt$ denotes the lisse extension (\defnref{defn:uncircumspectly}).
\end{cor}
\begin{proof}
  Given $X\in\Stk_S^G$, we may replace $F$ by its restriction to $\Lis_{[X/G]}$ in view of \propref{prop:entropionize}.
  Then we conclude by applying \thmref{thm:gawkiness}.
\end{proof}


\section{Pro-approximation in weaves}
\label{sec:acyclic}

\ssec{Sheaf cohomology}

  We now turn our attention to cohomology theories represented by a sheaf in some category of coefficients.
  That is, let $\D : \Sch^\op \to \InftyCat$ be a Nisnevich sheaf of \inftyCats and consider its lisse extension $\D^\lisse : \Stk^\op \to \InftyCat$ as in \defnref{defn:uncircumspectly}, so that
  \[ \D^\lisse(\sX) = \lim_{(T,t) \in \Lis_\sX} \D(T) \]
  for any $\sX \in \Stk$.
  We assume that $\D$ is lax symmetric monoidal (hence so is $\D^\lisse$), so that $\D^\lisse(\sX)$ is symmetric monoidal for every $\sX \in \Stk$, with monoidal unit denoted $\un_\sX$.

  In this situation, \thmref{thm:gawkiness} yields:

  \begin{cor}\label{cor:limoid}
    Assume that the sheaf of \inftyCats $\D$ is \emph{locally $\A^1$-invariant}\footnote{%
      I.e., $\id \to \pi_*\pi^*$ is invertible for the projection $\pi : X \times \A^1 \to X$ and every $X \in \Sch$.
    }.
    Let $G$ and $\{U_\nu\}_\nu$ be as in \notatref{notat:islandlike}.
    For every $X \in \Stk^G_S$ and $\sF \in \D^\lisse([X/G])$, the canonical morphism of spectra
    \begin{equation*}
      \RGamma([X/G], \sF) \to \lim_\nu \RGamma(X \fibprod_S^G U_\nu, \sF)
    \end{equation*}
    is invertible.
  \end{cor}
  \begin{proof}
    Given $X \in \Stk_S^G$ and $\sF \in \D^\lisse([X/G])$, consider the presheaf $F : \LisStk_{[X/G]}^\op \to \Spt$ defined by the assignment
    \begin{equation*}
      (T, t : T \to [X/G]) \mapsto \RGamma(T, \sF) := \Maps_{\D^\lisse(T)}(\un_T, t^*(\sF)).
    \end{equation*}
    This is lisse-extended from its restriction $F|_{\Lis_{[X/G]}}$, which is an $\A^1$-invariant Nisnevich sheaf by assumption.
    The claim now follows from \thmref{thm:gawkiness}.
  \end{proof}

  In this section, our goal is to address the analogue of \corref{cor:limoid} at the level of hypercohomology, i.e., invertibility of the morphisms
  \begin{equation*}
    \H^i([X/G], \sF) \to \lim_\nu \H^i(X \fibprod_S^G U_\nu, \sF)
  \end{equation*}
  for $i \in \Z$.
  This will require us to pass to a slightly more involved setup.

\ssec{Weaves}

  Let $\D$ be a \emph{weave} on $\Sch$ in the sense of \cite{Weaves,Weavelisse}, and consider its lisse extension $\D^\lisse$.
  By \cite[\S 7]{Weavelisse}, this defines a weave on $\Stk$ in which all morphisms are shriekable.

  Informally speaking, the weave $\D^\lisse$ amounts to a collection of \inftyCats $\D^\lisse(\sX)$ for every $\sX \in \Stk$, adjoint pairs of functors $(f^*,f_*)$ and $(f_!,f^!)$ for every morphism $f$ in $\Stk$, and a homotopy coherent system of various compatibilities between these operations.
  For example, we have the base change and projection formulas, Poincaré duality isomorphisms $f^! \simeq f^*\vb{\Omega_f}$ when $f$ is smooth\footnote{%
    Here $\vb{\cE}$ denotes the Thom twist by a finite locally free sheaf (or perfect complex) $\cE$, see \cite[\S 2.6]{Weaves}.
    An orientation of $\D$ in the sense of \cite[\S 2.7]{Weaves} determines Thom isomorphisms $\vb{\cE} \simeq \vb{\cO^{\oplus r}} \simeq (r)[2r]$, where $r$ is the (virtual) rank of $\cE$ and $(r) := \vb{\cO^{\oplus r}}[-2r]$ is the Tate twist.
  }, and ``forget supports'' isomorphisms $f_! \simeq f_*$ for $f$ proper representable.

  We will assume that $\D$ is \emph{topological} as in \cite[\S 2]{Weaves}, meaning that it satisfies homotopy invariance for vector bundles and localization for closed-open decompositions.
  Sometimes, notably in \corref{cor:outcourt}, we will also assume that $\D$ satisfies \emph{topological invariance}, i.e., that for a finite radicial surjection $f$, the functor $f^*$ is an equivalence (cf. \cite{topinv}).
  When $\D$ also satisfies continuity, this implies that that $f^*$ is an equivalence for any universal homeomorphism.
  These conditions persist to the lisse extension $\D^\lisse$.

  The classical examples are the Betti and étale weaves, which we will specialize to in \secref{sec:betet}.
  There are also various flavours of motivic weaves, such as the weave of sheaves of motivic spectra which we will come to in \secref{sec:sh}.

  \ssec{The cohomological t-structure}
  \label{ssec:acyclic/t}

  We define the cohomological t-structure on a topological weave $\D$.

  \begin{defn}
    Let $X \in \Sch$.
    Denote by $\D(X)_{\ge n} \sub \D(X)$ the full subcategory generated under colimits and extensions by objects of the form $a_!a^!(\un)(q)[n]$, where $a : T \to X$ is a smooth morphism from a scheme and $q\in\Z$, and by $\D(X)_{\le n} \sub \D(X)$ the full subcategory spanned by $\sF \in \D(X)$ for which the spectrum of derived global sections $\RGamma(T, \sF(q))$ is $n$-coconnective for all $q\in\Z$ and all smooth $X$-schemes $T$.
    The pair $(\D(X)_{\ge 0}, \D(X)_{\le -1})$ of orthogonal subcategories defines a t-structure on $\D(X)$ by \cite[Prop.~1.4.4.11]{LurieHA}.
  \end{defn}

  We say that $\sF \in \D(X)$ is \emph{$n$-connective}, resp. \emph{$n$-coconnective}, if it belongs to $\D(X)_{\ge n}$, resp. $\D(X)_{\le n}$.
  We say that $\sF$ is \emph{eventually connective}, resp. \emph{eventually coconnective}, if it belongs to $\D(X)_{>-\infty} = \bigcup_n \D(X)_{\ge n}$, resp. $\D(X)_{<\infty} = \bigcup_n \D(X)_{\le n}$.

  Note that, for a morphism $f : X' \to X$ in $\Sch$, the functor $f^*$ is right t-exact, i.e., preserves connectivity.
  If $f$ is smooth, $f^* \simeq f^! \vb{-\Omega_f}$ has a left adjoint $f_! \vb{\Omega_f}$.
  By the defintions, the latter is right t-exact.
  Hence by adjunction $f^*$ is also \emph{left} t-exact in this case, i.e., also preserves coconnectivity.

  We extend the cohomological t-structure to stacks as follows:

  \begin{prop}
    Let $\sX \in \Stk$.
    There exists a unique t-structure on the stable \inftyCat $\D^\lisse(\sX)$ such that $\sF \in \D^\lisse(\sX)$ belongs to $\D^\lisse(\sX)_{\le n}$, resp. $\D^\lisse(\sX)_{\ge n}$, if and only if for every $(T,t : T \to \sX) \in \Lis_\sX$, the object $t^*(\sF)$ belongs to  $\D(T)_{\le n}$, resp. $\D(T)_{\ge n}$.
  \end{prop}
  \begin{proof}
    By definition, we have equivalences
    \begin{equation*}
      \D^\lisse(\sX) \simeq \lim_t \D(T),
      \qquad \D^\lisse(\sX)_{\ge 0} \simeq \lim_t \D(T)_{\ge 0}
    \end{equation*}
    where the limits are taken over pairs $(T, t : T \to \sX) \in \Lis_\sX$ and the transition functors are $t^*$.
    Since the latter are t-exact, they restrict to left exact functors on the subcategories $\D(T)_{\ge 0}$.
    Moreover, each $\D(T)_{\ge 0}$ is a Grothendieck prestable \inftyCat by \cite[Prop.~C.1.4.1]{LurieSAG}.
    In this situation \cite[Prop.~C.3.2.4]{LurieSAG} implies that the limit $\D^\lisse(\sX)_{\ge 0}$ is also Grothendieck prestable and that the functors $t^* : \D^\lisse(\sX)_{\ge 0} \to \D(T)_{\ge 0}$ are left exact and jointly conservative.
    Passing back to stabilizations, it follows from \cite[Cor.~C.3.2.5, Prop.~C.1.4.1]{LurieSAG} that $\D^\lisse(\sX)$ admits a t-structure whose connective part is $\D^\lisse(\sX)_{\ge 0}$ and such that the functors $t^* : \D^\lisse(\sX) \to \D(T)$ are t-exact and jointly conservative.
    The latter implies that an object $\sF \in \D^\lisse(\sX)$ belongs to the coconnective part of the t-structure if and only if $t^*(\sF) \in \D(T)_{\le 0}$ for every $(T,t)\in\Lis_\sX$.
  \end{proof}
  
  \begin{thm}\label{thm:tstruct}
    Let $f : \sX' \to \sX$ be a morphism in $\Stk$.
    \begin{thmlist}
      \item\label{item:tstructure/starpull}
      The functor $f^*$ is right t-exact.
      If $f$ is smooth, $f^*$ is also left t-exact.

      \item
      The functor $f_*$ is left t-exact.

      \item
      If $f$ is smooth of relative dimension $d$, then $f^![-2d]$ is left t-exact and $f_![2d]$ is right t-exact.

      \item\label{item:tstruct/!}
      Suppose $\D$ satisfies topological invariance and continuity.
      If $f$ is of fibre dimension $\le d$ \sssecref{sssec:fibdim}, then $f^![-2d]$ is left t-exact and $f_![2d]$ is right t-exact.

      \item
      For every K-theory class $v \in \K(\sX)$ of virtual rank $r$, the shifted Thom twist $\vb{v}[-2r] : \D^\lisse(\sX) \to \D^\lisse(\sX)$ is t-exact.
    \end{thmlist}
  \end{thm}
  \begin{proof}
    That $f^*$ is right t-exact (so that $f_*$ is left t-exact by adjunction) follows easily from the case of schemes.
    We deduce by Poincaré duality that $f^*$ is also left t-exact when $f$ is smooth.

    The statement about Thom twists $\vb{v}[-2r]$ can be checked smooth-locally on $\sX$ (since $*$-inverse image along smooth morphisms is t-exact), so we may assume the K-theory class $v$ can be represented as a difference of finite locally frees.
    In this case we reduce to showing that $\vb{r}[-2r] = (r)$ (Tate twist) is t-exact.
    This is clear since it is evidently left t-exact, and admits a left adjoint $(-r)$ which is also left t-exact.

    For $f : \sX' \to \sX$ smooth of relative dimension $d$ we deduce that $f^! \simeq f^* \vb{\Omega_f}$ sends $n$-coconnective objects to $(n+2d)$-coconnective objects, and its left adjoint $f_!$ sends $n$-connective objects to $(n-2d)$-connective objects.
    
    If $f : \sX' \to \sX$ is a morphism of fibre dimension $\le d$, then the same holds assuming that $\D$ satisfies topological invariance and continuity.
    Indeed, take an object $\sF \in \D(\sX')_{\ge n}$ and let us show that $f_!(\sF)$ is $(n-2d)$-connective.
    Replacing $f$ by its base change along some smooth atlas $X \twoheadrightarrow \sX$, we may assume that $X = \sX$ is a scheme.
    By \cite[Prop.~B.3]{BachmannHoyois} it will suffice to show that $x^*f_!(\sF)$ is $(n-2d)$-connective for every point $x : \Spec(k(x)) \to X$.
    By the base change formula we may thus reduce further to the case where $X$ is the spectrum of a field $\kappa$, and by topological invariance and continuity that it is moreover perfect.
    By nilpotent invariance we may also assume that $\sX'$ is reduced.
    Then $\sX'$ admits a dense open which is smooth over $\Spec(\kappa)$ (take a smooth atlas by a scheme and take the image of the smooth locus of the latter), so by noetherian induction and the localization triangle we reduce to the case where $f : \sX' \to \Spec(\kappa)$ is smooth.
  \end{proof}
  
  Note that an immersion $i : \sY \hook \sX$ has fibre dimension $\le 0$, so \thmref{thm:tstruct}\itemref{item:tstruct/!} only yields that $i_!$ is right t-exact and $i^!$ is left t-exact.
  Nevertheless, we have:

  \begin{cor}\label{cor:tstructi^!}
    Suppose $\D$ satisfies topological invariance and continuity.
    Let $f : \sY \to \sX$ be a smooth morphism in $\Stk$ and $i : \sZ \hook \sY$ a closed immersion.
    Let $d_{\sY/\sX}$ and $d_{\sZ/\sX}$ denote the fibre dimensions \sssecref{sssec:fibdim} and write $c = d_{\sY/\sX}|_\sZ - d_{\sZ/\sX}$.\footnote{%
      We regard $c$ as a locally constant function on $\sZ$.
    }
    Then $i^!f^*[2c]$ is left t-exact.
  \end{cor}
  \begin{proof}
    By Poincaré duality for the smooth morphism $f$, we have $i^!f^*[2c] \simeq g^!\vb{-\Omega_f}[2c]$, where $g = f\circ i$.
    The claim is that for every $0$-coconnective $\sF \in \D^\lisse(\sX)$, $g^!(\sF)\vb{-\Omega_f}$ is $(-2c)$-coconnective.
    By \thmref{thm:tstruct}, $g^![-2d_{\sZ/\sX}]$ and $\vb{-\Omega_f}[2d_{\sY/\sX}]$ are left t-exact.
    In particular, $g^!(\sF)\vb{-\Omega_f}$ is $(2d_{\sZ/\sX}-2d_{\sY/\sX})$-coconnective if $\sF$ is $0$-coconnective.
  \end{proof}

\ssec{Acyclic morphisms}

  The following terminology is inspired by \cite[Exp.~XV, D\'ef.~1.7]{SGA4}:

  \begin{defn}
    Let $f : \sY \to \sX$ be a morphism in $\Stk$.
    We say that $f$ is:
    \begin{defnlist}
      \item
      \emph{$n$-acyclic}, for some $n\in\Z$, if the unit morphism $\sF \to f_*f^*(\sF)$ is $n$-coconnective\footnote{%
        A morphism is \emph{$n$-(co)connective} if its fibre is $n$-(co)connective.
      } for every $\sF \in \D^\lisse(\sX)_{<\infty}$;

      \item
      \emph{acyclic} if it is $n$-acyclic for every $n$.
    \end{defnlist}
  \end{defn}

  For example, any vector bundle projection is acyclic (by homotopy invariance).

  \begin{rem}
    Let $f : \sY \to \sX$ be a morphism in $\Stk$.
    \begin{defnlist}
      \item
      If $f$ is $n$-acyclic, then for every $\sF \in \D(\sX)_{<\infty}$ the map
      \[
        f^* : \RGamma(\sY, \sF) \to \RGamma(\sX, \sF)
      \]
      is injective on $\pi_{n+1}$ and bijective on $\pi_i$ for $i \ge n+2$.
      In other words, the map
      \[
        f^* : \H^i(\sY, \sF) \to \H^i(\sX, \sF)
      \]
      is injective is injective for $i=-n-1$ and bijective for $i\le -n-2$.

      \item
      If the cohomological t-structure on $\D^\lisse(\sX)$ is left-complete, then $f$ is $n$-acyclic if and only if $\sF \to f_*f^*(\sF)$ is an isomorphism for \emph{every} $\sF \in \D^\lisse(\sX)$.
      
      \item
      If the cohomological t-structure on $\D^\lisse(\sX)$ is right-separated (e.g. right-complete), then $f$ is acyclic if and only if $\sF \to f_*f^*(\sF)$ is an isomorphism for every $\sF \in \D^\lisse(\sX)_{<\infty}$.
    \end{defnlist}
  \end{rem}

\ssec{Pro-acyclic morphisms}

  \begin{defn}
    Let $\sX \in \Stk$ and let $\sY := \{\sY_\nu\}_\nu$ be a sequential diagram in $\Stk_\sX$ with structural morphisms $f_\nu : \sY_\nu \to \sX$.
    We say that $f := \{f_\nu\}_\nu : \sY \to \sX$ is:
    \begin{defnlist}
      \item
      \emph{$n$-acyclic} if for every $\sF \in \D^\lisse(\sX)_{<\infty}$, the canonical morphism in $\D^\lisse(\sX)$
      \[ \sF \to \lim_\nu f_{\nu,*}f_\nu^*(\sF) \]
      is $n$-coconnective;
      
      \item
       \emph{acyclic} if it is $n$-acyclic for all $n\in\Z$;
      
      \item
      \emph{$n$-pro-acyclic} if for every $\sF \in \D^\lisse(\sX)_{<\infty}$, there exists an index $\nu(n)$ such that the morphisms
      \[ \sF \to f_{\nu,*}f_\nu^*(\sF) \]
      are $n$-coconnective for all $\nu \ge \nu(n)$;

      \item
      \emph{pro-acyclic} if it is $n$-pro-acyclic for all $n\in\Z$.
    \end{defnlist}
    Since coconnectivity is stable under limits, $n$-pro-acyclic morphisms are $n$-acyclic.
  \end{defn}

  \begin{exam}
    If each $f_\nu : \sY_\nu \to \sX$ is $n$-acyclic for all $\nu$, then $f = \{f_\nu\}_\nu$ is $n$-pro-acyclic, hence $n$-acyclic.
    In fact, it suffices that for every $n\in\Z$ there exists an index $\nu(n)$ such that $f_\nu$ is $n$-acyclic for all $\nu\ge\nu(n)$.
  \end{exam}

  \begin{rem}\label{rem:proacycunravel}
    Suppose $f := \{f_\nu\}_\nu : \sY \to \sX$ is $n$-pro-acyclic.
    For every $\sF \in \D^\lisse(\sX)_{<\infty}$, it follows then from the definition of the cohomological t-structure that there exists an index $\nu(n)$ such that for all $\nu\ge\nu(n)$, the map
    \begin{equation}\label{eq:oyjgknnk}
      f_\nu^* : \RGamma(\sX, \sF) \to \RGamma(\sX, f_{\nu,*} f_\nu^* \sF ) \simeq \RGamma(\sY_\nu, \sF)
    \end{equation}
    has $n$-coconnective fibre, for all $\nu\ge\nu(n)$.
    In other words, \eqref{eq:oyjgknnk} is injective on $\pi_{n+1}$ and bijective on $\pi_i$ for $i \ge n+2$, i.e., that
    \begin{equation*}
      f_\nu^* : \H^i(\sX, \sF) \to \H^i(\sY_\nu, \sF)
    \end{equation*}
    is injective for $i=-n-1$ and bijective for $i\le -n-2$.
    Thus if $f$ is pro-acyclic, then for every integer $i\in\Z$ the maps
    \begin{equation}
      f_\nu^* : \H^i(\sX, \sF) \to \H^i(\sY_\nu, \sF)
    \end{equation}
    become invertible for sufficiently large $\nu$, and in particular
    \begin{equation}
      \H^i(\sX, \sF) \to \lim_\nu \H^i(\sY_\nu, \sF)
    \end{equation}
    is invertible for every $i\in\Z$.
  \end{rem}

  \begin{prop}\label{prop:spectroheliograph}
    Let $\sX \in \Stk$ and let $\{f_\nu : \sY_\nu \to \sX\}_\nu$ be a sequential diagram in $\Stk_\sX$.
    If $\{f_\nu\}_\nu$ is $n$-acyclic (resp. $n$-pro-acyclic), then for every smooth morphism $\sX' \to \sX$ in $\Stk$, so is the base change $\{f'_\nu : \sY_\nu \fibprod_\sX \sX' \to \sX'\}_\nu$.
  \end{prop}
  \begin{proof}
    Follows from the smooth base change formula and the fact that $*$-inverse image along smooth morphisms is t-exact.
  \end{proof}

\ssec{Pro-acyclicity of the Borel construction}

  We adopt again the setup of \secref{sec:borel}.
  In this subsection we will prove:

  \begin{prop}\label{prop:chrestomathy}
    Let $\D$ be a lisse-extended topological weave on $\Stk$, and assume $\D$ satisfies topological invariance and continuity.
    Then we have:
    \begin{thmlist}
      \item
      The morphism $\{U_\nu/G\}_\nu \to [S/G] = BG$ is pro-acyclic.
      
      \item
      For every $X\in\Stk^G_S$, the morphism $$\{X \fibprod^G_S U_\nu\}_\nu \to [X/G]$$ is pro-acyclic, where $X \fibprod^G_S U_\nu := [X/G] \fibprod_{BG} (U_\nu/G)$.
    \end{thmlist}
  \end{prop}

  \begin{cor}\label{cor:outcourt}
    Let $\D$ be as in \propref{prop:chrestomathy}.
    For every $X\in\Stk^G_S$ and every $\sF \in \D^\lisse(S)_{<\infty}$, the morphism
    \begin{equation}
      \H^i([X/G], \sF) \to \lim_\nu \H^i(X \fibprod_S^G U_\nu, \sF)
    \end{equation}
    is invertible for all $i\in\Z$.
    Moreover, there exists a sufficiently large index $\nu$ such that
    \begin{equation*}
      \H^i([X/G], \sF) \to \H^i(X \fibprod_S^G U_\nu, \sF)
    \end{equation*}
    is invertible.
  \end{cor}
  \begin{proof}
    Combine \propref{prop:chrestomathy} with \remref{rem:proacycunravel}.
  \end{proof}

  \begin{rem}\label{rem:spearwood}
    \corref{cor:outcourt} says in other words that for any eventually coconnective $\sF \in \D^\lisse([X/G])_{<\infty}$, the pro-system
    \begin{equation*}
      \{ \pi_i \RGamma(X \fibprod^G_S U_\nu, \sF) \}_\nu
    \end{equation*}
    is essentially constant for all $i\in\Z$.
    This may fail without the eventually coconnective hypothesis.
    For example, take $X=S=\Spec(k)$ with $k$ a field, $G=\bG_{m,k}$, $\D=\SH$ as in \secref{sec:sh}, $E = \mrm{KGL} \in \SH(k)$ the algebraic K-theory spectrum, and $\sF=a^*\mrm{KGL}$ where $a : B\bG_{m,k} \to \Spec(k)$.
    Then the pro-system
    \begin{equation*}
      \{ \pi_0 \RGamma(U_\nu/\bG_{m,k}, \mrm{KGL}) \}_\nu
      \simeq \{ \K_0(\P^\nu_k) \}_\nu
      \simeq \{ \bZ[t]/(t^\nu) \}_\nu
    \end{equation*}
    is not isomorphic to the constant pro-system $\{ \bZ[\![t]\!] \}$.
    We thank Marc Levine for pointing out this example.
  \end{rem}

  The proof of \propref{prop:chrestomathy} will use the following lemma:

  \begin{lem}\label{lem:insistingly}
    Suppose that $\D$ satisfies topological invariance and continuity.
    Let $\sX \in \Stk$ and let $\{j_\nu : \sU_\nu \to \sY_\nu\}_\nu$ be a sequential diagram of open immersions in $\Stk_\sX$ with $\sY_\nu$ smooth over $\sX$.
    Write $c_\nu := d_{\sY_\nu/\sX}|_{\sZ_\nu} - d_{\sZ_\nu/\sX}$, where $\sZ_\nu := \sY_\nu \setminus \sU_\nu$ are the reduced complements and $d_{\sY_\nu/\sX}$ and $d_{\sZ_\nu/\sX}$ are the fibre dimensions \sssecref{sssec:fibdim}.
    Assume that for every $c\in\Z$ there exists an index $\nu(c)$ for which $c_\nu \ge c$ for all $\nu\ge\nu(c)$.
    If $\{f_\nu : \sY_\nu \to \sX\}_\nu$ is pro-acyclic, then $\{g_\nu : \sU_\nu \to \sX\}_\nu$ is pro-acyclic.
  \end{lem}
  \begin{proof}
    By definition of the lisse extension and of the cohomological t-structure on $\D^\lisse(\sX)$, we may assume that $X=\sX$, $Y_\nu=\sY_\nu$, and $U_\nu=\sU_\nu$ are schemes.
    For every $\nu$ and every $\sF \in \D(X)_{<\infty}$ we have a commutative triangle
    \[\begin{tikzcd}[column sep=20]
      & \sF \ar{ld}\ar{rd} &
      \\
      f_{\nu,*}f_\nu^*(\sF) \ar{rr}
      & & g_{\nu,*}g_\nu^*(\sF)
    \end{tikzcd}\]
    where the horizontal arrow is induced by the unit $\id \to j_{\nu,*}j_\nu^*$.
    This gives rise to the exact triangle
    \[ \sK(f_\nu) \to \sK(g_\nu) \to \sK_\nu \]
    where $\sK(f_\nu)$, $\sK(g_\nu)$, and $\sK_\nu$ are the fibres of the left-hand diagonal, right-hand diagonal, and horizontal arrows respectively.
    It will thus suffice to show the following claim:
    \begin{enumerate}
      \item[$(\ast)$]
      For every integer $n\in\Z$ and every $\sF \in \D^\lisse(\sX)_{<\infty}$, choose an integer $c \ge (C-n)/2$; then $\sK_{\nu}$ is $n$-coconnective for all $\nu\ge \nu(c)$.
    \end{enumerate}
    By the localization triangle, we have
    \[ \sK_\nu \simeq f_{\nu,*} i_{\nu,*} i_\nu^! f_\nu^*(\sF) \]
    where $i_\nu : Z_\nu \to Y_\nu$ is the inclusion of the reduced complement of $U_\nu$.
    Let $C \in \Z$ such that $\sF$ is $C$-coconnective.
    For every fixed $n\in\Z$, if $c \ge (C-n)/2$ then $d_{Y_\nu/X}|_{Z_\nu} - d_{Z_\nu/X} \ge c$ for all $\nu \ge \nu(c)$.
    It follows by \thmref{thm:tstruct} and \corref{cor:tstructi^!} that $\sK_\nu$ is $(C-2c)$-coconnective.
    As $(C-2c) \le n$, $\sK_\nu$ is in particular $n$-coconnective for every $\nu \ge \nu(c)$.
  \end{proof}

  \begin{proof}[Proof of \propref{prop:chrestomathy}]
    Given $X \in \Stk^G_S$, consider the tower of open immersions
    \[
      \{X \fibprod_S^G U_\nu \to X \fibprod_S^G V_\nu\}_{\nu>0}
    \]
    over $[X/G]$, where the notation is as in \ref{notat:islandlike}.
    For every $\nu$, $X \fibprod_S^G V_\nu$ is the total space of a vector bundle over $[X/G]$, hence is acyclic (by homotopy invariance for $\D$).
    It will therefore suffice to check the condition on fibre dimensions in \lemref{lem:insistingly}.
    Since fibre dimension is stable under base change, this follows from the fact that $c_\nu = d_{V_\nu/S}|_{W_\nu} - d_{W_\nu/S}$ tends to $\infty$, where $W_\nu \sub V_\nu$ is the reduced complement of $U_\nu$ (see \remref{rem:borelfeatures}).
  \end{proof}


\section{Betti and étale (co)homology}
\label{sec:betet}

In this section we specialize the results of the previous section to the following weaves:
\begin{defnlist}
  \item\emph{Betti:}
  Suppose $k=\bC$.
  For every locally of finite type $k$-scheme $X$, let $\D(X) := \mrm{D}(X(\bC), \Lambda)$ denote the derived \inftyCat of sheaves of $\Lambda$-modules on the topological space $X(\bC)$, for some commutative ring $\Lambda$.
  This satisfies topological invariance by definition.

  \item\emph{Étale (torsion coefficients):}
  For every locally of finite type $k$-scheme $X$, let $\D(X) := \mrm{D}_\et(X, \Lambda)$ denote the derived \inftyCat of sheaves of $\Lambda$-modules on the small étale site of $X$, where $\Lambda$ is a commutative ring of positive characteristic $n$, with $n$ invertible in $k$.
  Topological invariance holds by \cite[Exp.~VIII, Thm.~1.1]{SGA4}.

  \item\emph{Étale (adic coefficients):}
  For every locally of finite type $k$-scheme $S$, let $\D(X)$ denote the limit of \inftyCats
  \[
    \lim_{n>0} \mrm{D}_\et(X, \Lambda/\mfr{m}^{n})
  \]
  where $\Lambda$ is a discrete valuation ring whose residue characteristic is invertible in $k$.
  Topological invariance follows from the case of torsion coefficients.
\end{defnlist}
In each case the unit $\un_X \in \D(X)$ is just the constant sheaf with coefficients in $\Lambda$.
These all satisfy étale descent, so the lisse extension $\D^\lisse$ is the unique étale sheaf on $\Stk$ which restricts to $\D$ on $\Sch$ (in particular, it coincides with the extension to stacks considered in \cite{LiuZheng}).
More generally the following discussion goes through for topological weaves which are oriented and satisfy topological invariance and continuity, and for which the unit $\un_X \in \D(X)$ lies in the heart of the cohomological t-structure for every $X \in \Sch$.

Let $S$, $G$, and $\{U_\nu\}_\nu$ be as in \notatref{notat:islandlike}.
Corollaries~\ref{cor:limoid} and \ref{cor:outcourt} specialize to:

\begin{cor}\label{cor:balanophore}
  For every $X \in \Stk_S^G$ and $\sF \in \D^\lisse([X/G])$, there are canonical isomorphisms
  \[ \RGamma([X/G], \sF) \simeq \lim_\nu \RGamma(X \fibprod^G_S U_\nu, \sF). \]
  Moreover, if $\sF \in \D^\lisse([X/G])_{<\infty}$ is eventually coconnective, then there are canonical isomorphisms
  \[ \H^s([X/G], \sF) \simeq \lim_\nu \H^s(X \fibprod^G_S U_\nu, \sF) \]
  for every $s\in\Z$.
\end{cor}

Taking coefficients in the constant sheaf $\Lambda_{BG} = \un_{BG} \in \D^\lisse(BG)$ we deduce:

\begin{cor}\label{cor:lumberjack}
  For every $X \in \Stk_S^G$, there are canonical isomorphisms
  \begin{equation*}
    \H^s_G(X)
    \simeq \lim_\alpha \H^s(X \times^G U_\alpha)
  \end{equation*}
  for all $s\in\bZ$.
\end{cor}
\begin{proof}
  The object $f^*(\Lambda_{BG}) = \Lambda_{[X/G]} \in \D^\lisse([X/G])$ is $0$-coconnective (\thmref{thm:tstruct}\itemref{item:tstructure/starpull}).
\end{proof}

For $X \in \Stk^G$, define the equivariant Borel--Moore homology spectrum (relative to the base $S$)
\[
  \Chom^{\BM,G}(X; \Lambda) = \RGamma([X/G], f^!(\Lambda_{BG}))
\]
where $f : [X/G] \to [S/G] = BG$ is the projection.

\begin{cor}\label{cor:spongicolous}
  We have
  \[ \Chom^{\BM,G}(X; \Lambda) \simeq \lim_\nu \CBM(X \fibprod^G_S U_\nu; \Lambda)(-d_\nu+g)[-2d_\nu+2g], \]
  where $d_\nu$ is the relative dimension of $U_\nu \to S$ and $g$ is the relative dimension of $G \to S$.
\end{cor}
\begin{proof}
  By \corref{cor:balanophore} we have
  \[ \RGamma([X/G], f^!(\Lambda_{BG})) \simeq \lim_\nu \RGamma(X \fibprod^G_S U_\nu, q_\nu^* f^!(\Lambda_{BG})) \]
  where $q_\nu : X \fibprod^G_S U_\nu \to [X/G]$ is the base change of $U_\nu/G \to BG$.
  The latter are smooth of relative dimension $d_\nu$.
  By the Poincaré duality isomorphisms $q_\nu^* \simeq q_\nu^!(-d_\nu)[-2d_\nu]$ and $\Lambda_{BG} \simeq a_{BG}^!(\Lambda_S)(g)[2g]$, where $a_{BG} : BG \to S$, the right-hand side is identified with the limit of the (shifted and Tate twisted) Borel--Moore chains on $X \fibprod^G_S U_\nu$ as claimed.
\end{proof}

We define the equivariant Borel--Moore homology groups by
\[
  \H^{\BM,G}_s(X; \Lambda) = \pi_s \RGamma([X/G], f^!(\Lambda_{BG})) \simeq \H^{-s}([X/G], f^!(\Lambda_{BG}))
\]
for $s\in\Z$.

\begin{cor}\label{cor:tuberculation}
  If $X$ is of finite dimension (in the sense of \cite[0AFL]{Stacks} or \cite[Eq.~(11.14)]{LaumonMoretBailly}), then
  \[ \H^{\BM,G}_s(X; \Lambda) \simeq \lim_\nu \H^\BM_{s+2d_\nu-2g}(X \fibprod^G_S U_\nu; \Lambda)(-d_\nu+g). \]
  for every $s\in\Z$.
\end{cor}
\begin{proof}
  If $X$ is of dimension $\le d$, then $f : [X/G] \to BG$ is of relative dimension $\le d$, so $f^!(\Lambda_{BG})$ is $2d$-coconnective by \thmref{thm:tstruct}.
  We conclude by the second part of \corref{cor:balanophore}.
\end{proof}


\section{Generalized cohomology theories}
\label{sec:sh}

We now consider the stable motivic homotopy weave $\D=\SH$.
In fact, the following results hold for arbitrary topological weaves with the same proofs, but for concreteness we just record the universal case.

Given a scheme $X$, let $\SH(X)$ denote the stable \inftyCat of motivic spectra over $X$ (see e.g. \cite[App.~C]{HoyoisLefschetz}), and consider the lisse extension $\SH^\lisse(-)$ with respect to $*$-inverse image.

Let $S$, $G$, and $\{U_\nu\}_\nu$ be as in \notatref{notat:islandlike}.
By \corref{cor:limoid} we deduce the following, a vast generalization of \cite[Thm.~12.9]{sixstack}.

\begin{cor}\label{cor:propylidene}
  For every $X \in \Stk_S^G$ and $\sF \in \SH^\lisse([X/G])$, there are canonical isomorphisms
  \[ \RGamma([X/G], \sF) \simeq \lim_\nu \RGamma(X \fibprod^G_S U_\nu, \sF). \]
\end{cor}

The six functor formalism on $\SH(-)$ persists to the lisse extension $\SH^\lisse(-)$ by \cite{ChowdhuryDAngelo,Weavelisse}.
In particular, for any locally of finite type morphism $f$ in $\Stk$ one has the adjoint pair of functors $(f_!, f^!)$.

Given a motivic spectrum $E \in \SH(S)$, let $E_{BG} = E|_{BG}$ denote its $*$-inverse image in $\SH^\lisse(BG)$.
For $X \in \Stk_S^G$, define the equivariant Borel--Moore homology spectrum (relative to the base $S$)
\[
  \Chom^{\BM,G}(X; E) = \RGamma([X/G], f^!(E_{BG}))
\]
where $f : [X/G] \to [S/G] = BG$ is the projection.
We denote by $\vb{n} \simeq (n)[2n]$ the Thom twist by the trivial bundle of rank $n$.

\begin{cor}\label{cor:Orchestia}
  For every $X \in \Stk_S^G$ there is a canonical isomorphism
  \[ \Chom^{\BM,G}(X; E) \simeq \lim_\nu \CBM(X \fibprod^G_S U_\nu; E)\vb{-\Omega_{U_\nu/S}+\Omega_{G/S}}. \]
  If $E$ is oriented, then moreover
  \[ \Chom^{\BM,G}(X; E) \simeq \lim_\nu \CBM(X \fibprod^G_S U_\nu; E)\vb{-d_\nu+g}, \]
  where $d_\nu$ (resp. $g$) is the relative dimension of $U_\nu \to S$ (resp. $G \to S$).
\end{cor}
\begin{proof}
  Follows from \corref{cor:propylidene} as in the proofs of Corollaries~\ref{cor:spongicolous} and \ref{cor:tuberculation}, using the (unoriented) Poincaré duality isomorphisms
  \begin{equation*}
    q_\nu^* \simeq q_\nu^! \vb{-\Omega_{U_\nu/S}},
    \quad E_{BG} \simeq a_{BG}^!(E)\vb{\Omega_{G/S}},
  \end{equation*}
  where $\Omega_{U_\nu/S}$ is the ($G$-equivariant) relative cotangent sheaf of $U_\nu \to S$, and similarly for $\Omega_{G/S}$, and $a_{BG} : BG \to S$ is the projection.
\end{proof}

Suppose $k$ is a field, $S=\Spec(k)$.
If $k$ has characteristic exponent $e$, then $\SH[1/e]$ (and more generally $\D[1/e]$ for any topological weave $\D$) satisfies topological invariance by \cite{topinv} as well as continuity.
The cohomological t-structure of \ssecref{ssec:acyclic/t} is the homotopy t-structure in this case.
By \corref{cor:outcourt} we obtain:

\begin{cor}
  If $E \in \SH(k)_{<\infty}$ is eventually coconnective, then for every $X \in \Stk_S^G$ there are canonical isomorphisms
  \[ \H^s([X/G], E)\inv{e} \simeq \lim_\nu \H^s(X \fibprod^G_S U_\nu, E)\inv{e} \]
  for every $s\in\Z$.
\end{cor}

Consider the equivariant Borel--Moore homology groups
\[
  \H^{\BM,G}_s(X; E) = \pi_s \RGamma([X/G], f^!(E_{BG})) \simeq \H^{-s}([X/G], f^!(E_{BG}))
\]
for $s\in\Z$, where $f : [X/G] \to BG$.
The argument of \corref{cor:tuberculation} shows:

\begin{cor}\label{cor:omniparent}
  If $X$ is of finite dimension and $E \in \SH(S)_{<\infty}$ is eventually coconnective, then 
  \[ \H^{\BM,G}_s(X; E)\inv{e} \simeq \lim_\nu \H^\BM_s(X \fibprod^G_S U_\nu; E)\inv{e}\vb{-\Omega_{U_\nu/S}+\Omega_{G/S}} \]
  for every $s\in\Z$.
  If $E$ is oriented, then moreover
  \[ \H^{\BM,G}_s(X; E)\inv{e} \simeq \lim_\nu \H^\BM_{s+2d_\nu-2g}(X \fibprod^G_S U_\nu; E)\inv{e}(-d_\nu+g) \]
  for every $s\in\Z$.
\end{cor}

Let $\Lambda$ be a commutative ring in which $e$ is invertible, and let $E=\Lambda^{\mot} \in \SH(k)$ be the $\Lambda$-linear motivic cohomology spectrum.
Since the latter is eventually coconnective, we may apply \corref{cor:omniparent}.
Combining this with the comparison between motivic Borel--Moore homology of schemes and higher Chow groups (see \cite[Prop.~19.18]{MazzaVoevodskyWeibel} and \cite[Cor.~8.12]{CisinskiDegliseIntegral}), we thus deduce:

\begin{cor}\label{cor:lactoscope}
  For every algebraic space $X$ of finite type over $k$ with $G$-action, there are canonical isomorphisms
  \[
    \H^{\BM,G}_{s+2n}(X; \Lambda^\mot)(-n)
    \simeq \on{A}^G_{n}(X, s) \otimes \Lambda
  \]
  for all $n,s\in\bZ$, where on the right-hand side are the $G$-equivariant higher Chow groups of $X$ \cite[\S 2.7]{EdidinGraham}.
\end{cor}

Similarly, one has
\[
  \H^{\BM}_{s+2n}(\sX; \Lambda^\mot)(-n)
  \simeq \on{A}_n(\sX, s) \otimes \Lambda,
\]
where $\sX = [X/G]$ and the right-hand side is defined in \cite[\S 5.3]{EdidinGraham} or \cite{Kresch}.
We expect this comparison to generalize to all Artin stacks $\sX$ of finite type over $k$ with affine stabilizers, cf. \cite{BaePark}.

\begin{rem}
  The right-hand side of \corref{cor:omniparent} indicates a definition of $G$-equivariant (higher) Chow--Witt groups in terms of the Borel construction.
  By construction, the resulting theory identifies with the generalized Borel--Moore homology theory associated with the Milnor--Witt motivic cohomology spectrum (in a manner parallel to \corref{cor:lactoscope}).
\end{rem}


\section{Algebraic K-theory}
\label{sec:k}

Consider the Nisnevich sheaf of spectra $\K : \Asp^{\op} \to \Spt$ which sends an algebraic space $X$ of finite type over $k$ to its Bass--Thomason--Trobaugh K-theory spectrum (see e.g. \cite[Def.~2.6, Rem.~2.15]{KhanKstack}).
Write $\KH : \Asp^{\op} \to \Spt$ for its $\A^1$-invariant version, defined by
\[ \KH(X) = \colim_{[n]\in\bDelta^\op} \K(X \times \A^n) \]
for every $X \in \Asp$, see e.g. \cite[\S 4.2]{KhanKstack}.
The canonical map $\K(X) \to \KH(X)$ is invertible when $X$ is regular.

We study the lisse extensions $\K^\lisse$ and $\KH^\lisse$.
Note that with rational coefficients, the presheaf $\K(-)_\Q : \Asp^\op \to \Spt$ sending $X \in \Asp$ to its rationalized K-theory spectrum $\K(X) \otimes \Q$, satisfies étale descent (see e.g. \cite[Thm.~5.1]{KhanKstack}).
Thus its lisse extension $\K(-)_\Q^\lisse$ is the unique extension of $\K_\Q$ to an étale sheaf on Artin stacks; in particular, it coincides with the construction $\K^\et(-)_\Q$ in \cite[\S 5.2]{KhanKstack}.

Let $S$, $G$, and $\{U_\nu\}_\nu$ be as in \notatref{notat:islandlike}.
Applying \corref{cor:hemiataxy} to the presheaf $\KH|_{\Sch^\lci}$ yields:

\begin{cor}
  For every algebraic space $X$ of finite type over $S$ with $G$-action, there is a canonical isomorphism
  \[ \KH^\lisse([X/G]) \simeq \lim_\nu \KH(X \fibprod^G_S U_\nu). \]
\end{cor}

\begin{rem}
  More generally, we get the same computation for any localizing invariant of stable \inftyCats which satisfies $\A^1$-invariance.
  For example, this also applies to topological K-theory (over the complex numbers) and periodic cyclic homology in characteristic zero.
\end{rem}

\begin{cor}
  If $X$ is regular, then moreover
  \[ \K^\lisse([X/G]) \simeq \lim_\nu \K(X \fibprod^G_S U_\nu). \]
\end{cor}

Suppose the base ring $k$ is noetherian.
Since lci morphisms are of finite Tor-amplitude, G-theory (= algebraic K-theory of coherent sheaves) defines a presheaf of spectra $\G : \Asp^{\lci,\op} \to \Spt$ on the category of algebraic spaces of finite type over $k$ and lci morphisms (see e.g. \cite[\S 3]{KhanKstack}).
Thus \corref{cor:hemiataxy} yields:

\begin{cor}\label{cor:kerflummox}
  For every algebraic space $X$ of finite type over $S$ with $G$-action, there is a canonical isomorphism
  \[ \G^\lisse([X/G]) \simeq \lim_\nu \G(X \fibprod^G_S U_\nu). \]
\end{cor}

\begin{defn}
  With notation as above, we define \emph{Borel-type $G$-equivariant K-theory}, \emph{KH-theory}, and \emph{G-theory} by
  \begin{align*}
    \K^{G,\lisse}(X) &:= \K^\lisse([X/G]),\\
    \KH^{G,\lisse}(X) &:= \KH^\lisse([X/G]),\\
    \G^{G,\lisse}(X) &:= \G^\lisse([X/G])
  \end{align*}
  for all algebraic spaces $X$ of finite type over $S$ with $G$-action.
\end{defn}

\begin{prop}\label{prop:Gconn}
  For every algebraic space $X$ of finite type over $S$ with $G$-action, the spectrum $\G^{G,\lisse}(X)$ is connective.
  If $X$ is regular, then the same holds for $\K^{G,\lisse}(X)$.
\end{prop}
\begin{proof}
  Recall that the claim is true for $G$ trivial (i.e., $\G(X)$ is connective), and note that the second claim follows from the first since $\K(T) \simeq \G(T)$ for every $(T,t) \in \Lis_X$ when $X$ (hence $T$) is regular.
  
  For every $s > 0$, we have using \corref{cor:kerflummox} the canonical surjections
  \begin{equation}\label{eq:iprkebtm}
    \pi_{-s} \G^{G,\lisse}(X)
    \twoheadrightarrow \lim_\nu \pi_{-s} \G(X\fibprod^G U_\nu) \simeq 0
  \end{equation}
  with kernel $\lim^1_\nu \pi_{-s+1} \G(X \fibprod^G U_\nu)$.
  We claim that the pro-system
  \begin{equation}\label{eq:gkxhjuaj}
    \{ \pi_{-s+1} \G(X \fibprod^G U_\nu) \}_\nu
  \end{equation}
  satisfies the Mittag--Leffler condition for every $s>0$.
  This will imply that the $\lim^1$ terms vanish, hence $\pi_{-s}\G^{G,\lisse}(X)$ all vanish.

  By \remref{rem:borelfeatures} the transition map for $\mu>\nu$ is the composite of the restriction map along an open,
  \begin{equation*}
    \pi_{-s+1} \G(X \fibprod^G U_\mu)
    \to \pi_{-s+1} \G(X \fibprod^G (U_\nu \fibprod_S V_{\mu-\nu}))
  \end{equation*}
  and pull-back along the zero section of a vector bundle,
  \begin{equation*}
    \pi_{-s+1} \G(X \fibprod^G (U_\nu \fibprod_S V_{\mu-\nu}))
    \to \pi_{s+1} \G(X \fibprod^G U_\nu).
  \end{equation*}
  The former is surjective if $s=1$ and bijective if $s>1$ (by the long-exact localization sequence and the connectivity of G-theory of algebraic spaces), and the latter is bijective for all $s$ (by homotopy invariance).
  Thus \eqref{eq:gkxhjuaj} has surjective transition maps and in particular satisfies the Mittag--Leffler condition.
\end{proof}

We now prove \thmref{thmX:grr}.
From now on, we take $S$ to be the spectrum of the base commutative ring $k$ (so that $G$ is a group scheme over $k$).
We begin with the following spectrum-level formulation of equivariant Grothendieck--Riemann--Roch:

\begin{thm}\label{thm:grrspt}
  Suppose $k$ is regular noetherian.
  Let $X$ be an algebraic space of finite type over $k$ with $G$-action.
  Then there is a canonical isomorphism of spectra
  \begin{equation*}
    \G^{G,\lisse}(X)_\Q \simeq \prod_{i\in\Z} \Chom^{\BM,G}(X; \Q^\mot)\vb{-i}.
  \end{equation*}
  Moreover, this isomorphism commutes with equivariant proper push-forwards and equivariant quasi-smooth Gysin pull-backs.
\end{thm}

We will come to the proof at the end of this section, after recording several consequences of \thmref{thm:grrspt}.

First, taking homotopy groups and combining with \corref{cor:lactoscope}, we deduce:

\begin{cor}\label{cor:antiforeignism}
  Suppose $k$ is a field.
  Then for every algebraic space $X$ of finite type over $k$ with $G$-action, there are canonical isomorphisms
  \[
    \G^{G,\lisse}_s(X)_\Q
    \simeq \prod_{i\in\Z} \on{A}^G_{i}(X, s)_\Q
  \]
  for all $s\in\Z$.
  Moreover, these isomorphisms commute with equivariant proper push-forwards and equivariant quasi-smooth Gysin pull-backs.
\end{cor}

Rationally, we can also show that the $\lim{}\!^1$ obstruction vanishes:

\begin{cor}\label{cor:gimmel}
  Suppose $k$ is regular noetherian.
  Then for every algebraic space $X$ of finite type over $k$ with $G$-action, the canonical morphisms
  \begin{equation*}
    \G_s^{G,\lisse}(X)_\Q
    \to \lim_\nu \G_s(X \fibprod^G U_\nu)_\Q
  \end{equation*}
  are bijective for all $s\in\Z$.
\end{cor}
\begin{proof}
  In view of \corref{cor:kerflummox} there is a canonical surjection
  \begin{equation*}
    \G_s^{G,\lisse}(X)_\Q \twoheadrightarrow \lim_\nu \G_s(X \fibprod^G U_\nu)_\Q.
  \end{equation*}
  Under the isomorphisms of \thmref{thm:grrspt} this is identified with the product over $i\in\Z$ of the maps
  \begin{equation*}
    \pi_s \Chom^{\BM,G}(X;\Q^\mot)\vb{-i}
    \twoheadrightarrow \lim_\nu \pi_s \CBM(X\fibprod^G U_\nu;\Q^\mot)\vb{-i-d_\nu+g},
  \end{equation*}
  since products commute with $\pi_s$ and $\lim$.
  Each of these maps are bijective by \corref{cor:omniparent}, so the claim follows.
\end{proof}

\begin{cor}
  Suppose $k$ is a field and $G$ is a smooth affine algebraic group over $k$.
  Then for every quasi-projective scheme $X$ over $k$ with linearized $G$-action, there is a canonical isomorphism
  \[
    \G_0^G(X)^\wedge_{I_G}
    \simeq \prod_{i\in\Z} \on{A}^G_{i}(X)_\Q
  \]
  where the left-hand side is the completion of the $G$-equivariant G-theory of $X$ at the augmentation ideal $I_G \sub \K_0(BG)$.
\end{cor}
\begin{proof}
  Krishna shows in \cite[Thm.~9.10]{KrishnaCompletion} that under the assumptions, the completion $\G_0^G(X)^\wedge_{I_G}$ agrees with the right-hand side of \corref{cor:gimmel} for $s=0$.
\end{proof}

\begin{rem}
  In \cite{CarlssonJoshua} G.~Carlsson and R.~Joshua showed (under some technical hypotheses) that the right-hand side of \corref{cor:kerflummox}, and hence $\G^{G,\lisse}(X)$, agrees with the ``Adams completion'' of $\G^G(X) = \G([X/G])$ with respect to the augmentation map $\K(BG) = \K^G(S) \to \K(S)$.

  At the level of homotopy groups, A.~Krishna studied in \cite{KrishnaCompletion} the question of bijectivity of the map (the ``Atiyah--Segal completion problem'')
  \begin{equation}\label{eq:unforewarnedness}
    \G^G_s(X)^\wedge_{I_G} \to \lim_\nu \G_s(X \fibprod^G U_\nu)
  \end{equation}
  for $X$ smooth quasi-projective.
  He showed that this holds when $X$ is moreover projective and $G$ is connected split reductive, but that it may fail for $X$ non-projective (even with $G=\bG_m$).
  In general, we can say that in the smooth but non-projective case the Atiyah--Segal map \eqref{eq:unforewarnedness} is invertible with rational coefficients.
  Indeed, rationally both sides are isomorphic to $\prod_{i\in\Z} \on{A}^G_i(X,s)_\Q$: the right-hand side by Corollaries~\ref{cor:antiforeignism} and \corref{cor:gimmel}, and the left-hand side by Krishna's version of equivariant GRR in \cite{KrishnaRR}.
\end{rem}

We begin the proof of \thmref{thm:grrspt}, which involves some stable motivic homotopy theory.
Let $\KGL \in \SH(k)$ denote the algebraic K-theory spectrum over $k$ (see \cite[\S 13.1]{CisinskiDegliseBook}).
Set
\begin{equation*}
  \KGL^\lisse_\sX := a^*(\KGL),
  \quad
  \Q^{\mot,\lisse}_\sX := a^*(\Q^\mot)
  \in \SH^\lisse(\sX)
\end{equation*}
for every $\sX \in \Stk$ with structural morphism $a : \sX \to \Spec(k)$, where $\Q^\mot \in \SH(k)$ is the rational motivic cohomology spectrum as in \corref{cor:lactoscope}.
We begin by generalizing the Adams decomposition of $\KGL_\Q$ (see \cite{Riou}, \cite[\S 14.1]{CisinskiDegliseBook}) to stacks.

\begin{prop}\label{prop:pyruvate}\leavevmode
  \begin{thmlist}
    \item
    For every $\sY \in \Stk$, there is a canonical isomorphism
    \begin{equation}
      \KGL^\lisse_{\sY,\Q} \simeq \bigoplus_{i\in\Z} \Q^{\mot,\lisse}_\sY\vb{i}.
    \end{equation}
    If $\sY$ is smooth then there is also a canonical isomorphism
    \begin{equation}
      \KGL^\lisse_{\sY,\Q} \simeq \prod_{i\in\Z} \Q^{\mot,\lisse}_\sY\vb{i}.
    \end{equation}

    \item
    For every morphism $f : \sX \to \sY$ in $\Stk$ with $\sY$ smooth, there is a canonical isomorphism
    \begin{equation}
      f_*f^!(\KGL^\lisse_{\sY,\Q}) \to \prod_{i\in\Z} f_*f^!(\Q^{\mot,\lisse}_\sY)\vb{i}
    \end{equation}
    in $\SH^\lisse(\sY)$.
  \end{thmlist}
\end{prop}
\begin{proof}
  Consider the canonical morphisms
  \begin{equation*}
    \bigoplus_{i\in\Z} \Q^{\mot}\vb{i} \to \KGL_\Q \to \prod_{i\in\Z} \Q^{\mot}\vb{i}
  \end{equation*}
  in $\SH(k)$.
  The first is invertible by \cite[Thm.~5.3.10]{Riou}.
  The composite is also invertible: for every smooth $k$-scheme $X$ and $s \ge 0$, it induces an isomorphism on $\H^{-s}(X, -)$ since
  \begin{equation*}
    \H^{-s}(X, \Q^\mot\vb{i}) \simeq \H^{2i-s}(X; \Q(i)) \simeq \Gr_\gamma^i \K_{s}(X)_\Q
  \end{equation*}
  vanishes for $i<0$ and $i \gg 0$.
  The second isomorphism uses the fact that $X$ is smooth over $k$ and hence regular (see \cite[Cor.~14.2.14]{CisinskiDegliseBook}).
  
  The isomorphisms in the first claim follow by $*$-inverse image along $a : \sY \to \Spec(k)$, which commutes with colimits (resp. limits when $\sY$ is smooth).
  The second claim then follows since $f^!$ and $f_*$ commute with limits.
\end{proof}

\begin{proof}[Proof of \thmref{thm:grrspt}]
  For every $X$ as in the statement we have by \propref{prop:pyruvate} a canonical isomorphism $f_*f^!(\KGL^\lisse_{BG,\Q}) \simeq \prod_{i\in\Z} f_*f^!(\bQ^{\mot,\lisse}_{BG})\vb{i}$ where $f : [X/G] \to BG$ is the projection.
  Formation of derived global sections commutes with limits, so as $X$ varies this gives the canonical isomorphism
  \begin{equation*}
    \Chom^{\BM,G}(-; \KGL_\Q)
    \simeq \prod_{i\in\Z} \Chom^{\BM,G}(-; \Q^{\mot})\vb{i}.
  \end{equation*}
  The claim follows by combining this with the canonical isomorphism
  \begin{equation*}
    \G^{G,\lisse}(-)
    \simeq \Chom^{\BM,G}(-; \KGL)
  \end{equation*}
  obtained by lisse-extending the isomorphism $\G(-) \simeq \CBM(-; \KGL)$ of sheaves of spectra on $\Sch^\lci$ (see \cite{JinG}).
  The first displayed isomorphism commutes with equivariant proper push-forwards and equivariant quasi-smooth Gysin pull-backs by construction.
  The second displayed isomorphism commutes with equivariant proper push-forwards by the arguments of \cite[\S 3.1]{JinG} and with equivariant quasi-smooth Gysin pull-backs by \cite[\S 6.2]{KhanKstack}.
\end{proof}


\section{Algebraic bordism}
\label{sec:bordism}

Let $S=\Spec(k)$, $G$ an affine algebraic group over $k$, and $\{U_\nu\}_\nu$ as in \notatref{notat:islandlike}.
\corref{cor:Orchestia} applied to the algebraic cobordism spectrum $\MGL \in \SH(k)$ (and its Tate twists) shows that the associated equivariant Borel--Moore theory can be computed, at the level of spectra, by the Borel construction:

\begin{cor}\label{cor:semidiatessaron}
  For every $X\in \Stk_S^G$ and $r,s\in\Z$ there are canonical isomorphisms of spectra
  \[
    \Chom^{\BM,G}(X; \MGL) \simeq \lim_\nu \CBM(X\fibprod^G U_\nu; \MGL)\vb{-d_\nu+g}
  \]
  where $d_\nu=\dim(U_\nu)$ and $g=\dim(G)$.
\end{cor}

Recall that if $X$ is a quasi-projective $k$-scheme with linearized $G$-action, then each $X\fibprod^G U_\nu$ is a quasi-projective $k$-scheme (Remark~\ref{rem:unsainted}).
Thus in this case \corref{cor:semidiatessaron} computes the $G$-equivariant bordism of $X$ in terms of non-equivariant bordism of schemes.

When $k$ is a field of characteristic zero, there is an identification
\begin{equation*}
  \H^\BM_{2n}(X; \MGL)(-n)
  \simeq \Omega_n(X),
\end{equation*}
for all quasi-projective $k$-schemes $X$, of the $(2*,*)$-graded part of Borel--Moore homology with coefficients in $\MGL$, with the (lower) algebraic bordism theory defined by Levine and Morel (see \cite{LevineMorel,LevineMGL}).

Building on the theory $\Omega_*(-)$, Heller and Malag\'on-L\'opez defined in \cite{HellerMalagonLopez} a (lower) $G$-equivariant algebraic bordism theory by the formula
\begin{equation}\label{eq:HML}
  \Omega^{G,\mrm{HML}}_n(X) := \lim_\nu \Omega_{n+d_\nu-g}(X\fibprod^G U_\nu)
\end{equation}
for $X$ quasi-projective with linearized $G$-action.
By \corref{cor:semidiatessaron}, the canonical surjection from the homotopy groups of a homotopy limit to the limit of its homotopy groups reads in this case
\begin{equation}\label{eq:waff}
  \begin{multlined}
    \pi_0 \Chom^{\BM,G}(X; \MGL)\vb{-n}\\
    \twoheadrightarrow \lim_\nu \pi_0 \CBM(X\fibprod^G U_\nu; \MGL)\vb{-n-d_\nu+g}
    \simeq \Omega^{G,\mrm{HML}}_{n}(X)
  \end{multlined}
\end{equation}
for every $n\in\Z$.
Since $\MGL \in \SH(k)$ is not eventually coconnective, we cannot apply \corref{cor:omniparent} to deduce that the above map is invertible.
Nevertheless, we can show this holds after rationalization:

\begin{thm}\label{thm:mglq}
  Let $k$ be a field of characteristic zero, $G$ an affine algebraic group, and $X$ a quasi-projective $k$-scheme with linearized $G$-action.
  Then the canonical map
  \begin{equation*}
    \pi_0 \Chom^{\BM,G}(X; \MGL_\Q)\vb{-n}\\
    \twoheadrightarrow \Omega^{G,\mrm{HML}}_{n}(X)_\Q
  \end{equation*}
  is invertible for every $n\in\Z$.
\end{thm}

We begin with the analogue of \propref{prop:pyruvate} for $\MGL$.
For $\sX \in \Stk$ with structural morphism $a : \sX \to \Spec(k)$, set $\MGL^\lisse_\sX := a^*(\MGL) \in \SH^{\lisse}(\sX)$ and $\Q^{\mot,\lisse}_\sX := a^*(\Q^\mot)$.

\begin{prop}\label{prop:pyruvate MGL}\leavevmode
  \begin{thmlist}
    \item
    For every $\sY \in \Stk$, there is a canonical isomorphism
    \begin{equation}
      \MGL^\lisse_{\sY,\Q} \simeq \bigoplus_{i\ge 1} \Q^{\mot,\lisse}_\sY\vb{i}.
    \end{equation}
    If $\sY$ is smooth then there is also a canonical isomorphism
    \begin{equation}
      \MGL^\lisse_{\sY,\Q} \simeq \prod_{i\ge 1} \Q^{\mot,\lisse}_\sY\vb{i}.
    \end{equation}

    \item
    For every morphism $f : \sX \to \sY$ in $\Stk$ with $\sY$ smooth, there is a canonical isomorphism
    \begin{equation}
      f_*f^!(\MGL^\lisse_{\sY,\Q}) \to \prod_{i\ge 1} f_*f^!(\Q^{\mot,\lisse}_\sY)\vb{i}
    \end{equation}
    in $\SH^\lisse(\sY)$.
  \end{thmlist}
\end{prop}
\begin{proof}
  Recall that there is a canonical decomposition $\MGL_\Q \simeq \bigoplus_{i\ge 1} \Q\vb{i}$ in $\SH(k)$ (see \cite[Cor.~10.6]{NaumannSpitzweckOstvaer}).
  Hence one may argue exactly as in the proof of \propref{prop:pyruvate}, using the vanishing of $\Gr_\gamma^i\K_s(X)_\Q$ for $i \gg 0$.
\end{proof}

\begin{proof}[Proof of \thmref{thm:mglq}]
  For every finite type $k$-scheme $Y$ with action of an affine algebraic group $H$, we have natural isomorphisms
  \begin{equation}
    \Chom^{\BM,H}(Y; \MGL_\Q) \simeq \prod_{i\ge 1} \Chom^{\BM,H}(Y; \Q)\vb{i}
  \end{equation}
  by \propref{prop:pyruvate MGL} applied to $f : [Y/H] \to BH$.
  Under these isomorphisms, the surjection
  \begin{equation}\label{eq:waffq}
    \begin{multlined}
      \pi_0 \Chom^{\BM,G}(X; \MGL_\Q)\vb{-n}\\
      \twoheadrightarrow \lim_\nu \pi_0 \CBM(X\fibprod^G U_\nu; \MGL_\Q)\vb{-n-d_\nu+g}
      \simeq \Omega^{G,\mrm{HML}}_{n}(X)_\Q
    \end{multlined}
  \end{equation}
  is identified with the product over $i \ge 1$ of the maps
  \begin{equation*}
    \pi_0 \Chom^{\BM,G}(X;\Q^\mot)\vb{-n+i}
    \twoheadrightarrow \lim_\nu \pi_0 \CBM(X\fibprod^G U_\nu;\Q^\mot)\vb{-n+i-d_\nu+g},
  \end{equation*}
  since products commute with $\pi_0$ and $\lim$.
  Each of these maps are bijective by \corref{cor:omniparent}, so the claim follows.
\end{proof}

With integral coefficients, we suspect that \eqref{eq:waff} may not be invertible in general.\footnote{%
  For $X$ smooth projective and $G$ connected and split reductive, we expect that \eqref{eq:waff} is invertible by arguments similar to those in \cite[\S\S 6-7]{KrishnaCompletion}.
}
We propose that the ``correct'' definition\footnote{%
  Of course, one could take this as a definition, but by convention $\Omega^G_*(-)$ should be a ``geometrically'' defined theory as in \cite{LevineMorel}.
  More precisely, for every $n$ there should be geometrically defined sheaves of spectra on quasi-projective $k$-schemes, analogous to the Bloch cycle complexes, whose hypercohomologies compute $\Omega_n(-)$.
  (For smooth schemes $X$, such has been constructed in \cite[Thm.~3.4.1(ii)]{deloop3}, using finite quasi-smooth derived schemes over $X$ as generators.)
  Then the $G$-equivariant versions of these spectra should be defined by the Borel construction, i.e., via a homotopy limit as in \corref{cor:semidiatessaron}.
  Finally, $\Omega^G_n(-)$ should be defined as the $\pi_0$ of this spectrum-level $G$-equivariant bordism theory.
} of $\Omega^G_n(-)$ should satisfy
\begin{equation}
  \Omega^G_n(X) \simeq \pi_0 \Chom^{\BM,G}(X; \MGL)\vb{-n}.
\end{equation}
By the general properties of the construction $\CBM(-; E)$, for $E \in \SH(k)$, $\Omega^G_*(-)$ admits proper push-forwards and quasi-smooth Gysin pull-backs, and satisfies homotopy invariance and the projective bundle formula.
We will now show that it also satisfies the right-exact localization property\footnote{%
  Although it is claimed in \cite[Thm.~20]{HellerMalagonLopez} that the same holds for $\Omega^{G,\mrm{HML}}_*(-)$, there is a well-known gap in their proof.
}.
In \cite{virloc,constack} it is shown that it also satisfies the equivariant concentration and localization theorems.

\begin{thm}\label{thm:unchanted}
  Let $k$ be a perfect field, $G$ an affine algebraic group over $k$, and $X$ an algebraic space of finite type over $k$ with $G$-action.
  Then we have:
  \begin{thmlist}
    \item\label{item:metropolis}
    For every $G$-equivariant closed immersion $i : Z \to X$ with complementary open immersion $j : U \to X$, the localization sequence
    \[
      \begin{multlined}
        \H^{\BM,G}_{2n}(Z; \MGL)(-n)\\\quad\xrightarrow{i_*}
        \H^{\BM,G}_{2n}(X; \MGL)(-n) \xrightarrow{j^*}
        \H^{\BM,G}_{2n}(U; \MGL)(-n) \to 0.
      \end{multlined}
    \]
    is exact.
    \item\label{item:antirattler}
    For every integer $n\in\Z$, the spectrum $\Chom^{\BM,G}(X; \MGL)\vb{-n}$ is connective.
    That is, we have
    \[ \H^{\BM,G}_{s}(X; \MGL)(-r) = 0 \]
    for all $r,s\in\Z$ with $s<2r$.
  \end{thmlist}
\end{thm}

\begin{proof}
  The localization exact triangle
  \[
    \Chom^{\BM,G}(Z; \MGL) \to \Chom^{\BM,G}(X; \MGL) \to \Chom^{\BM,G}(U; \MGL)
  \]
  induces a long exact sequence in $\H^{\BM,G}_*(-; \MGL)(-n)$ for every $n \in \Z$.
  In particular we have the exact sequence
  \[
    \begin{multlined}
      \H^{\BM,G}_{2n}(Z; \MGL)(-n) \xrightarrow{i_*}
      \H^{\BM,G}_{2n}(X; \MGL)(-n) \xrightarrow{j^*}
      \H^{\BM,G}_{2n}(U; \MGL)(-n)\\\xrightarrow{\partial}
      \H^{\BM,G}_{2n-1}(Z; \MGL)(-n).
    \end{multlined}
  \]
  This shows that the second claim implies the first.

  We now demonstrate claim~\itemref{item:antirattler} in the case of $G$ the trivial group.
  If $X$ is a smooth scheme, then by Poincaré duality the claim is equivalent to the vanishing of $\H^q(X; \MGL(p))$ for $q > 2p$, see e.g. \cite[Thm.~B.1]{ChowTstr}.
  In general, the schematic locus defines a dense open $U \sub X$ (see \cite[Tag 06NH]{Stacks}).
  Since the field $k$ is perfect, the smooth locus $V$ of $U$ is a further dense open.
  Using the localization sequence
  \begin{equation*}
    \H^\BM_{s}(Z; \MGL)(-r)
    \to \H^\BM_{s}(X; \MGL)(-r)
    \to \H^\BM_{s}(V; \MGL)(-r),
  \end{equation*}
  where $Z \sub X$ is the reduced closed complement, we conclude by noetherian induction.

  For the case of general $G$, the proof is similar to that of \propref{prop:Gconn}.
  We claim that for every $n\in\Z$ the pro-system
  \begin{equation}\label{eq:spiculated}
    \{ \pi_0 \Chom^{\BM}(X \fibprod^G U_\nu; \MGL)\vb{-n-d_\nu} \}_\nu
  \end{equation}
  satisfies the Mittag--Leffler condition.
  Indeed, by \remref{rem:borelfeatures} the transition map for $\mu>\nu$ is by construction the composite of the restriction map along an open,
  \begin{equation*}
    \pi_0 \Chom^{\BM}(X \fibprod^G U_\mu; \MGL)\vb{-n-d_\mu}
    \to \pi_0 \Chom^{\BM}(X \fibprod^G (U_\nu \times V_{\mu-\nu}); \MGL)\vb{-n-d_\mu}
  \end{equation*}
  and Gysin pull-back along the zero section of a vector bundle,
  \begin{equation*}
    \pi_0 \Chom^{\BM}(X \fibprod^G (U_\nu \times V_{\mu-\nu}); \MGL)\vb{-n-d_\mu}
    \to \pi_0 \Chom^{\BM}(X \fibprod^G U_\nu; \MGL)\vb{-n-d_\nu}.
  \end{equation*}
  By claim~\itemref{item:metropolis} in the case of $G$ trivial, the former is surjective.
  By homotopy invariance, the latter is invertible.
  Thus the pro-system \eqref{eq:spiculated} has surjective transition maps and in particular satisfies the Mittag--Leffler condition.

  Using \corref{cor:semidiatessaron}, we have the canonical surjections
  \begin{equation}\label{eq:ginny}
    \pi_{-1} \Chom^{\BM,G}(X; \MGL)\vb{-n}\\
    \twoheadrightarrow \lim_\nu \pi_{-1} \CBM(X\fibprod^G U_\nu; \MGL)\vb{-n-d_\nu+g}
  \end{equation}
  with kernel $\lim^1_\nu \pi_0 \CBM(X \fibprod^G U_\nu; \MGL)\vb{-n-d_\nu}$.
  The latter vanishes by the Mittag--Leffler condition verified above.
  Hence in that case, Claim~\itemref{item:metropolis} for $G$ trivial implies that the target vanishes, hence so does the source.

  Now, by induction we see that the pro-system $\{\pi_{-s+1} \CBM(X \fibprod^G U_\nu; \MGL)\vb{\ast} \}_\nu$ vanishes for all $s>1$.
  Arguing by the Milnor exact sequence again, we have bijectivity of \eqref{eq:ginny} for all lower homotopy groups as well, and we conclude again by Claim~\itemref{item:metropolis} for $G$ trivial.
  This shows claim~\itemref{item:antirattler}.
\end{proof}

Combining with \thmref{thm:mglq}, we deduce that the theory of \cite{HellerMalagonLopez} satisfies the right-exact localization property with rational coefficients:

\begin{cor}\label{cor:mgllocq}
  Let $k$ be a field of characteristic zero, $G$ an affine algebraic group over $k$, and $X$ a quasi-projective $k$-scheme with linearized $G$-action.
  Then for every $G$-equivariant closed immersion $i : Z \to X$ with complementary open immersion $j : U \to X$, the localization sequence
  \[
    \begin{multlined}
      \Omega^{G,\mrm{HML}}_n(Z)_\Q\xrightarrow{i_*}
      \Omega^{G,\mrm{HML}}_n(X)_\Q \xrightarrow{j^*}
      \Omega^{G,\mrm{HML}}_n(U)_\Q \to 0.
    \end{multlined}
  \]
  is exact.
\end{cor}


\section{Categorification}
\label{sec:cat}

Let $\D : \Sch^\op \to \Pres$ be a Nisnevich sheaf with values in the \inftyCat $\Pres$ of presentable \inftyCats and left adjoint functors.
Denote by $\D^\lisse : \Stk^\op \to \InftyCat$ its lisse extension as in \defnref{defn:uncircumspectly}.
Given a morphism $f : \sX \to \sY$ we denote the induced functor by $f^* = \D(f) : \D(\sY) \to \D(\sX)$, and by $f_* : \D(\sX) \to \D(\sY)$ its right adjoint.

We will assume that $\D$ satisfies the following two properties:
\begin{defnlist}
  \item
  \emph{Local $\A^1$-invariance}: for every $X \in \Sch$, the unit morphism
  \[ \id \to \pi_*\pi^* \]
  is fully faithful, where $\pi : X\times\A^1 \to X$ is the projection.
  In other words, $\pi^* : \D(X) \to \D(X\times\A^1)$ is fully faithful.

  \item\label{item:vibrioid}
  \emph{Smooth base change formula}: for every cartesian square in $\Sch$
  \[\begin{tikzcd}
    X' \ar{r}{g}\ar{d}{u}
    & Y' \ar{d}{v}
    \\
    X \ar{r}{f}
    & Y,
  \end{tikzcd}\]
  the base change transformation
  \[\begin{tikzcd}
    v^* f_*
    \xrightarrow{\mrm{unit}} g_* g^* v^* f_*
    \simeq g_* u^* f^* f_*
    \xrightarrow{\mrm{counit}} g_* u^*
  \end{tikzcd}\]
  is invertible.
\end{defnlist}
For example, we may take $\D = \SH$ (see \secref{sec:sh}) or more generally any topological weave in the sense of \cite{Weaves}.

Fix $S$, $G$, and $\{U_\nu\}_\nu$ as in \notatref{notat:islandlike}.
For every $X\in\Stk_S^G$ we consider the square
\begin{equation}\label{eq:coppet}
  \begin{tikzcd}
    X \fibprod_S U_\nu \ar{r}{p_\nu}\ar{d}{v_\nu}
    & X \ar{d}{u}
    \\
    X \fibprod^G_S U_\nu \ar{r}{q_\nu}
    & {[X/G]}
  \end{tikzcd}
\end{equation}
where $p_\nu$ and $q_\nu$ are the projections, and $u$ and $v_\nu$ are the quotient maps.

\begin{thm}\label{thm:recool}
  For every $\sF \in \D^\lisse([X/G])$, the unit maps induce a canonical isomorphism
  \begin{equation*}
    \sF \to \lim_\nu q_{\nu,*} q_\nu^{*} (\sF)
  \end{equation*}
  in $\D^\lisse([X/G])$.
\end{thm}

\begin{proof}
  For a fixed $\sF \in \D^\lisse([X/G])$, the presheaf
  $$F : \LisStk_{[X/G]}^\op \to \D^\lisse([X/G])$$
  sending $(T, t : T \to [X/G]) \mapsto t_*t^*(\sF)$ is, by construction, lisse-extended from its restriction to $\Lis_{[X/G]}$.
  Thus the claim follows from \propref{prop:moneyed} applied to $F|_{\Lis_{[X/G]}}$.
\end{proof}

Consider the right Kan extension of $\D^\lisse$ to ind-objects, so that
\[
  \D^\lisse(X\fibprod_S \{U_\nu\}_\nu) \simeq \lim_\nu \D^\lisse(X \fibprod_S U_\nu),
  \quad \D^\lisse(X \fibprod^G_S \{U_\nu\}_\nu) \simeq \lim_\nu \D^\lisse(X \fibprod^G_S U_\nu)
\]
where the transition functors are $*$-inverse image.
We have the induced functors
\begin{align}
  &p^* = (p_\nu^*)_\nu : \D^\lisse(X) \to \D^\lisse(X \fibprod_S \{U_\nu\}_\nu),\nonumber\\
  &q^* = (q_\nu^*)_\nu : \D^\lisse([X/G]) \to \D^\lisse(X \fibprod^G_S \{U_\nu\}_\nu).\label{eq:haggada}
\end{align}

\begin{cor}\label{cor:estovers}
  The functor \eqref{eq:haggada} is fully faithful.
\end{cor}
\begin{proof}
  The functor $q^*$ admits as right adjoint $(\sF_\nu) \mapsto \lim_\nu q_{\nu,*} \sF_\nu$, so fully faithfulness amounts to invertibility of the unit map
  \[
    \sF \to \lim_\nu q_{\nu,*} q_\nu^{*} (\sF)
  \]
  for all $\sF \in \D([X/G])$, which is the assertion of \thmref{thm:recool}.
\end{proof}

We say that the group scheme $G$ is \emph{Nisnevich-special} if the quotient morphism $S \twoheadrightarrow [S/G] = BG$ admits Nisnevich-local sections, i.e., if every étale $G$-torsor is Nisnevich-locally trivial.
For example, this includes special group schemes in the sense of Serre such as $\GL_{n,S}$.

\begin{thm}\label{thm:soteriologic}
  If $\D$ satisfies étale descent, or $G$ is Nisnevich-special, then the squares \eqref{eq:coppet} induce a cartesian square of \inftyCats
  \[\begin{tikzcd}
    \D^\lisse([X/G]) \ar{r}{q^*}\ar{d}{u^*}
    & \D^\lisse(X \fibprod^G_S \{U_\nu\}_\nu) \ar{d}{v^*}
    \\
    \D^\lisse(X) \ar{r}{p^*}
    & \D^\lisse(X \fibprod_S \{U_\nu\}_\nu).
  \end{tikzcd}\]
\end{thm}

We will prove \thmref{thm:soteriologic} at the end of this section.

\begin{rem}
  If $\D$ only satisfies Nisnevich descent and $G$ is not Nisnevich-special, then one still has a cartesian square of \inftyCats
  \[\begin{tikzcd}
    \D^\lisse([X/G]) \ar{r}{q^*}\ar{d}{u^*}
    & \D^\lisse(X \fibprod^G_S \{U_\nu\}_\nu) \ar{d}{v^*}
    \\
    \D^\lisse(Y) \ar{r}{p^*}
    & \D^\lisse(Y' \fibprod^G_S \{U_\nu\}_\nu),
  \end{tikzcd}\]
  where $u : Y \twoheadrightarrow [X/G]$ is any smooth morphism with Nisnevich-local sections and $Y' \to Y$ is the $G$-torsor classified by $Y \to [X/G] \to BG$, since $u^*$ is conservative in this case.
\end{rem}

\begin{lem}\label{lem:dissyllabize}
  With notation as above, suppose that $X$ is a scheme.
  Then the cartesian square
  \[ \begin{tikzcd}
    X \fibprod_S \{U_\nu\}_\nu \ar{r}{p}\ar{d}{v}
    & X \ar{d}{u}
    \\
    X \fibprod^G_S \{U_\nu\}_\nu \ar{r}{q}
    & {[X/G]}
  \end{tikzcd} \]
  satisfies the smooth base change formula.
  That is, the natural transformation
  \[
    u^* q_*
    \xrightarrow{\mrm{unit}} p_* p^* u^* q_*
    \simeq p_* v^* q^* q_*
    \xrightarrow{\mrm{counit}} p_*v^* 
  \]
  is invertible.
\end{lem}
\begin{proof}
  This is the limit over $n$ of the natural transformations
  \[
    u^* q_{\nu,*}
    \to p_{\nu,*} v_\nu^*,
  \]
  associated to the squares \eqref{eq:coppet}.
  Using descent for the \v{C}ech nerve of $u : X \twoheadrightarrow [X/G]$ and its base change $v_\nu$, which we denote $X_\bullet$ and $Y_\bullet$ respectively, \cite[Chap.~1, 2.6.4]{GaitsgoryRozenblyum} implies that this is map is in turn the limit of the corresponding natural transformations for all the squares
  \[
    \begin{tikzcd}
      Y_{m+1} \ar{r}{p_m}\ar{d}{d^i}
      & X_{m+1} \ar{d}{d^i}
      \\
      Y_m \ar{r}{q_m}
      & X_m
    \end{tikzcd}
  \]
  where the horizontal arrows are base changed from $q$ and $p$ and the vertical arrows $d^i$ are the face maps (for $0\le i\le m$).
  By the smooth base change formula for schemes \itemref{item:vibrioid}, these are invertible for all $m$ and all $i$.
\end{proof}

\begin{lem}\label{lem:superremuneration}
  Suppose given a commutative square of \inftyCats
  \[\begin{tikzcd}
    \cC \ar{r}{q^*}\ar{d}{u^*}
    & \cC' \ar{d}{v^*}
    \\
    \cD \ar{r}{p^*}
    & \cD'
  \end{tikzcd}\]
  where $p^*$ and $q^*$ are fully faithful with respective right adjoints $p_*$ and $q_*$, the base change transformation
  \[
    u^* q_*
    \xrightarrow{\mrm{unit}} p_* p^* u^* q_*
    \simeq p_* v^* q^* q_*
    \xrightarrow{\mrm{counit}} p_*v^*
  \]
  is invertible, and $v^*$ is conservative.
  Then the essential image of $q^*$ is spanned by objects $c' \in \cC'$ for which $v^*(c')$ belongs to the essential image of $p^*$.
\end{lem}
\begin{proof}
  Note that an object $c' \in \cC'$ belongs to the essential image of $q^*$ if and only if the counit $q^*q_*(c') \to c'$ is invertible.
  Indeed, the condition is clearly sufficient.
  Conversely, suppose $c' \simeq q^*(c)$ for an object $c \in \cC$.
  By the adjunction identities, the composite
  \[
    q^*(c)
    \xrightarrow{\mrm{unit}} q^*q_*q^*(c)
    \xrightarrow{\mrm{counit}} q^*(c)
  \]
  is the identity.
  Since $q^*$ is fully faithful, the first arrow is invertible.
  It follows that the second arrow is also invertible.

  Now since $v^*$ is conservative, invertibility of the counit $q^*q_*(c') \to c'$ is equivalent to invertibility of
  \[
    \mrm{counit}:
    p^*p_*v^*(c')
    \simeq p^*u^*q_*(c')
    \simeq v^*q^*q_*(c')
    \xrightarrow{\mrm{counit}} v^*(c')
  \]
  where we have used the base change isomorphism.
  As in the first paragraph, since $p^*$ is fully faithful this is equivalent to the condition that $v^*(c')$ belongs to the essential image of $p^*$.
\end{proof}

\begin{proof}[Proof of \thmref{thm:soteriologic}]
  Given $(T,t) \in \Lis_{[X/G]}$, we may form the base change of the squares \eqref{eq:coppet} along $t : T \to [X/G]$ to get
  \[ \begin{tikzcd}
    T' \fibprod_S \{U_\nu\}_\nu \ar{r}{p_T}\ar{d}{v_T}
    & T' \ar{d}{u_T}
    \\
    T \fibprod_{BG} \{U_\nu/G\}_\nu \ar{r}{q_T}
    & T
  \end{tikzcd} \]
  where $T' \to T \simeq [T'/G]$ is the $G$-torsor classified by $T \to [X/G] \to BG$.
  By definition of $\D^\lisse$, the square in question is the limit over $(T,t)$ of the squares
  \[\begin{tikzcd}
    \D^\lisse(T) \ar{r}{q_T^*}\ar{d}{u_T^*}
    & \D^\lisse(T \fibprod_{BG} \{U_\nu/G\}_\nu) \ar{d}{v^*}
    \\
    \D^\lisse(T') \ar{r}{p^*}
    & \D^\lisse(T' \fibprod_S \{U_\nu\}_\nu).
  \end{tikzcd}\]
  We may therefore replace $X$ by $T'$ and thereby assume that $X$ is a scheme.

  By \corref{cor:estovers} the upper horizontal arrow is fully faithful.
  The same holds for the lower horizontal arrow (note that $\{U_\nu\}_\nu$ also serves as a Borel construction for the trivial group).
  This implies that the square is cartesian on mapping spaces.
  Essential surjectivity of the functor
  \[
    \D^\lisse([X/G])
    \to \D^\lisse(X) \fibprod_{\D^\lisse(X \fibprod_S \{U_\nu\}_\nu)} \D^\lisse(X \fibprod^G_S \{U_\nu\}_\nu)
  \]
  then follows from \lemref{lem:superremuneration} in view of the base change formula $u^* q_* \simeq p_* v^*$ (\lemref{lem:dissyllabize}) and the conservativity of $v^*$ (since $u$ and hence $v$ admits Nisnevich-local sections).
\end{proof}



\bibliographystyle{halphanum}

Institute of Mathematics, Academia Sinica, 10617 Taipei, Taiwan

Max-Planck-Institut für Mathematik, Vivatsgasse 7, 53111 Bonn, Germany

\end{document}